\documentclass[review]{elsarticle}
\biboptions{sort&compress}

\usepackage{lineno,hyperref}

\usepackage{mathtools,amsfonts,amsthm,amssymb}
\usepackage{subcaption,float,xcolor}

\newtheorem{thm}{Theorem}[section]

\newtheorem{lem}[thm]{Lemma}

\newtheorem{cor}[thm]{Corollary}

\allowdisplaybreaks

\newtheoremstyle{hdef}{1em}{0em}{}{}{\bfseries}{.}{.5em}{\thmname{#1}\thmnumber{ #2}\thmnote{ (\hspace{-.01pt}{#3})}}
\theoremstyle{hdef}

\newtheorem{dfn}[thm]{Definition}

\newtheorem{rem}[thm]{Remark}

\makeatletter
\newtheoremstyle{premark}{1em}{0em}{
\addtolength{\@totalleftmargin}{1.5em}
\addtolength{\linewidth}{-1.5em}
\parshape 1 1.5em \linewidth}{}{\scshape}{.}{.5em}{}
\makeatother

\theoremstyle{premark}

\DeclareMathOperator{\dif}{d}

\DeclareMathOperator{\Var}{Var}

\newcommand{\bN}{{\mathbb N}}

\newcommand{\bR}{{\mathbb R}}

\newcommand{\fC}{{\mathfrak C}}

\newcommand{\bs}{{\backslash}}
\renewcommand{\ss}{\subset}
\renewcommand{\a}{\alpha}
\renewcommand{\b}{\beta}
\renewcommand{\c}{\gamma}

\modulolinenumbers[5]

\journal{Journal of \LaTeX\ Templates}

\begin{document}

\begin{frontmatter}

\title{Approximation of Stieltjes ordinary differential equations\tnoteref{mytitlenote}}
\tnotetext[mytitlenote]{The authors were partially supported by Xunta de Galicia, project ED431C 2019/02, 
and by project MTM2016-75140-P of MINECO/FEDER (Spain).}

\author{
 Francisco J. Fern\'andez\\
 \normalsize\emph{ e-mail:} fjavier.fernandez@usc.es \\
F. Adri\'an F. Tojo\\
\normalsize\emph{ e-mail:} fernandoadrian.fernandez@usc.es\\
\normalsize \textit{Instituto de Ma\-te\-m\'a\-ti\-cas, Facultade de Matem\'aticas,} \\ \normalsize\textit{Universidade de Santiago de Com\-pos\-te\-la, Spain.}
}





\begin{abstract}

This work is devoted to the obtaining of a new numerical scheme based on quadrature formulae for the 
Lebesgue-Stieltjes integral for the approximation of Stieltjes ordinary differential equations. This 
novel method allows us to numerically approximate models based on Stieltjes ordinary 
differential equations for which no explicit solution is known. We prove several theoretical results 
related to the consistency, convergence and stability of the numerical method. We also obtain 
the explicit solution of the Stieltjes linear ordinary differential equation, and use it to 
validate the numerical method. Finally, we present some numerical results that we have 
obtained for a realistic population model based on a Stieltjes differential equation.
\end{abstract}


\end{frontmatter}

\linenumbers

\section{Introduction}

In this work we present a numerical method in order to approximate the solution of a Stieltjes differential equation of the type
\begin{equation}\label{eq:ec1}
\left\{\begin{array}{l}
x_g'(t)=f(t,x(t)),\; \text{for $g$-almost every $t\in [0,T)$}, \\
x(0)=x_0,
\end{array}\right.
\end{equation}
where, $x_g'$ is the Stieltjes derivative with respect to a left-continuous non decreasing function $g$. That is, given
$x:[0,T]\rightarrow \mathbb{R}$, we define, for each $t \in [0,T]\setminus C_g$, 
$x_g'(t)$ as the following limit in case it exists
\begin{equation}
u_g'(t):=\begin{dcases}
 \lim_{s\to t} \frac{x(s)-x(t)}{g(s)-g(t)}, & \text{if}\; t \notin D_g, 
\\ 
 \frac{x(t^+)-x(t)}{g(t^+)-g(t)}, & \text{if}\; t \in D_g,
\end{dcases}
\end{equation}
where $D_g$
denotes the set of discontinuities of $g$. In this particular case,
\begin{equation}
D_g=\{s \in \mathbb{R}:\; g(s^+)-g(s)>0\},
\end{equation} 
and
\begin{equation}
C_g=\{s \in \mathbb{R}:\; 
g \text{ is constant on } (s-\varepsilon,s+\varepsilon) \text{ for some } \varepsilon\in{\mathbb R}^+\}.
\end{equation}
While defining equation~\eqref{eq:ec1} for `$g$-almost every $t\in [0,T)$' we are implicitly considering the Lebesgue-Stieltjes measure space $([0,T],\mathcal{M}_g,\mu_g)$, where $\mathcal{M}_g$ is the $\sigma$-algebra and $\mu_g$ the measure constructed in an analogous fashion to the classical Lebesgue measure, where the length of $[a,b)$ is given by $\mu_g([a,b))=
g(b)-g(a)$. The interested reader may refer to \cite{POUSO2015} for details concerning this measure space. The theoretical study of this kind of derivatives and their applications appear, for instance, in
 \cite{POUSO2015,POUSO2017,POUSO2018,FRIGON2019,MaTo}. 

 \label{sh} As stated in \cite[Theorem 7.3]{POUSO2017}, in the case $g:[0,T]\rightarrow [0,\infty)$ is increasing, left-continuous and continuous at $0$, and $f:[0,T]\times \mathbb{R} \rightarrow \mathbb{R}$ satisfies
\begin{itemize}
\item[(H1)] $f(\cdot,x)$ is $g$-measurable for every $x \in \mathbb{R}$;
\item[(H2)] $f(\cdot,x_0)\in \mathcal{L}^1_g([0,T))$;
\item[(H3)] there exists $L\in \mathcal{L}^1_g([0,T),[0,\infty))$ such that for $g$-almost every $t \in [0,T)$ 
and every $x,y\in \mathbb{R}$ we have that
\begin{equation}
|f(t,x)-f(t,y)|\leq L(t) |x-y|;
\end{equation}
\end{itemize}
then, problem (\ref{eq:ec1}) has a unique solution in the space $\mathcal{BC}_g([0,T])$ of bounded $g$-continuous functions $u:[0,T]\rightarrow \mathbb{R}$, that is, the solution $u$ satisfies, for every $t_0\in[0,T]$,
\begin{equation}
\forall \epsilon>0,\; \exists \delta>0\,:\; [t \in [0,T],\; 
|g(t)-g(t_0)|<\delta]\Rightarrow |u(t)-u(t_0)|<\epsilon.
\end{equation} 
 $\mathcal{BC}_g([0,T])$ is a Banach space with the supremum norm --see \cite[Theorem~3.4]{POUSO2017}. Furthermore, the solution of problem~\eqref{eq:ec1} is the unique fixed point of the operator
\begin{equation}
\begin{array}{rcl}
F:\mathcal{BC}_g([0,T]) & \rightarrow & \mathcal{BC}_g([0,T]), \\ 
x & \rightarrow & F(x),
\end{array}
\end{equation}
where, given $t \in [0,T]$, 
\begin{equation}
F(x)(t)=x_0 + \int_{[0,t)} f(s,x(s))\, \dif\mu_g;
\end{equation}
that is, the solution of problem (\ref{eq:ec1}) is such that
\begin{equation} \label{eq:ec2}
x(t)=x_0+ \int_{[0,t)} f(s,x(s))\, \dif\mu_g,\; \forall t \in [0,T].
\end{equation}

Furthermore, from \cite[Lemma 7.2]{POUSO2017} and
\cite[Theorem 5.4]{POUSO2015}, the solution will belong to the space $\mathcal{AC}_g([0,T])$ of $g$-absolutely continuous functions, that is, of those functions $u:[0,T]\rightarrow \mathbb{R}$ such that, for every
$\epsilon>0$, there exists $\delta>0$ satisfying that, if $\{(a_n,b_n)\}_{n \in \mathbb{N}}$ is a collection of pairwise-disjoint open intervals such that
\begin{equation}
\sum_{n=1}^N |g(b_n)-g(a_n)|<\delta,
\end{equation}
then, 
\begin{equation}
\sum_{n=1}^N |f(b_n)-f(a_n)|<\epsilon.
\end{equation}
It is precisely expression (\ref{eq:ec2}) what motivates the approximation based on quadrature formulae for the Lebesgue-Stieltjes integral which we introduce in Section~\ref{FC}. We will see that, in order to obtain error bounds, it will be necessary to impose additional conditions on the regularity of the function $f$ and the solution of problem~\eqref{eq:ec1}.

In order to conveniently organize this work, in Section~\ref{FC} we obtain some 
numerical quadrature formulae for approximating the Lebesgue-Stieltjes integral, 
in Section~\ref{DNM} we present a predictor-corrector method based on the 
quadrature formulae obtained in Section~\ref{FC}. In Section~\ref{EA} we analyze mathematically
the consistency, convergence and stability of the numerical method derived in 
Section~\ref{DNM}. In order to validate the numerical method, in Section~\ref{GLE} we obtain 
the explicit solution of the general linear equation of Stieltjes type. Finally, in 
Section~\ref{NS}, we present some numerical results that we have obtained for 
the general linear equation and for a realistic silkworm population model 
based on a Stieltjes differential equation.

\section{Quadrature formulae for the Lebesgue-Stieltjes integral}\label{FC}

We now introduce some convenient notation. Given an increasing left-continuous function $g:[a,b]\rightarrow \bR$, 
we define $\Delta^+ g:[a,b)\rightarrow \bR$ as $\Delta^+ g(t)=g(t^+)-g(t)$. In the same way we define 
$\Delta^-h(t)=h(t^-)-h(t)$ whenever the left limit of $h$ exists at $t$. Clearly, $g$ is continuous at $t_0\in[a,b)$ 
if and only if $\Delta^+ g(t_0)=0$.
We have that
\begin{equation}0\le\sum_{t\in[a,b)}\Delta^+ g(t)\le g(b)-g(a),\end{equation}
so $g$ has a countable number of discontinuities, say those in $D_g=\{d_k\}_{k \in \Lambda}$, where $\Lambda\subset\bN$. If we define the bounded increasing function $g^B:[a,b)\to[0,+\infty)$ as
\begin{equation}
g^B(t)=\sum_{s\in[a,t)}\Delta^+ g(s)=\Delta^+ g(a) \chi_{(a,b]}(t)+
\sum_{k\in\Lambda}\Delta^+ g(d_k) \chi_{(d_k,b]}(t), 
\end{equation}
it is clear that $g^C:[a,b)\to\bR$, given by $g^C(t):=g(t)-g^B(t)$, is bounded, increasing and continuous. We say $g^C$ is the \emph{continuous part} of $g$ and $g^B$ is the \emph{jump part} of $g$.

As we foretold in the previous section, the numerical method we propose to approximate the solution of the differential problem 
\eqref{eq:ec1} in its integral form (\ref{eq:ec2}) will be based on the approximation of the Lebesgue-Stieltjes integral. We start this section by proving a result that will allow us to interpret the integral in (\ref{eq:ec2}) in terms of a Kurzweil-Stieltjes integral for which it will be possible to establish quadrature formulae.

\begin{lem}\label{lemfi} Let $g:[a,b]\rightarrow [0,\infty)$ be an increasing left-continuous function and $f \in \mathcal{L}^1_g([a,b))$. Then,
	\begin{equation}
	\int_{[a,b)} f\, \dif\mu_g = \int_a^b f(s) \dif g(s)= \int_a^b f(s) \dif g^C(s)+\sum_{s\in[a,b)}f(s)\Delta^+ g(s),
	\end{equation}
	where, in the right hand side, we consider a Kurzweil-Stieltjes integral. Furthermore, if $\{d_k\}_{k \in \Lambda}$ is the set of discontinuities of $g$ in $(a,b)$, we have that
	\begin{equation}
	\int_{[a,b)} f\, \dif\mu_g=\int_a^b f(s) \dif g^C(s)+ f(a) \Delta^+ g(a)+
	\sum_{k\in \Lambda}
	f(d_k)\Delta^+ g(d_k).
	\end{equation}

\end{lem}

\begin{proof}
	This Lemma is an immediate consequence of \cite[Theorems 6.12.3 and 6.3.13]{ANTUNES2019}. That is, since $f$ is Lebesgue-Stieltjes integrable, by \cite[Theorem 6.12.3]{ANTUNES2019}, it is Kurzweil-Stieltjes integrable as well and, furthermore,
	\begin{equation}
	\begin{array}{rcl}
	\displaystyle
	\int_{[a,b)} f \, \dif\mu_g &=&
	\displaystyle
	\int_{\{a\}} f\, \dif\mu_g + \int_{(a,b)} f \, \dif\mu_g, \\ 
\displaystyle
\int_{(a,b)} f\,\dif\mu_g
	&=&\displaystyle \int_a^b f(s) \, \dif g(s)- f(a)\Delta^+ g(a)- f(b)\Delta^- g(b).
	\end{array}
	\end{equation}
	Now, due to the fact that $g$ is left-continuous, $\Delta^- g(b)=0$ and, since $\mu_g(\{a\})=\Delta^+ g(a)$, we obtain
	\begin{equation}
	\int_{[a,b)} f \, \dif\mu_g = \int_a^b f(s) \, \dif g(s).
	\end{equation}
	Finally, by \cite[Theorem 6.3.13]{ANTUNES2019},
	\begin{equation}
	\begin{aligned}
	\int_a^b f(s) \, \dif g(s)=& \int_a^b f(s) \, \dif g^C(s)+
	f(a)\Delta^+ g(a)+f(b) \Delta g^-(b)\\ &+\sum_{k\in\Lambda}f(d_k) \Delta g(d_k).
	\end{aligned}
	\end{equation}
	Since $g$ is left-continuous we have, in particular, that
 $\Delta^- g(b)=\Delta^- g(d_k)=0$ for every $k \in \mathbb{N}$, 
and the desired result follows.
\end{proof}

In Lemma~\ref{lem1} we will see that, under certain regularity hypotheses on $f$ and $g$, we can obtain error estimates for the quadrature formula for a point and the trapeze formula.

\begin{lem} \label{lem1} Let us assume $f \in BV([a,b])\cap \mathcal{L}^1_g([a,b))$ and 
	$g:[a,b]\rightarrow[0,\infty)$ is increasing and left-continuous. Furthermore, assume $g^C$ is
	\emph{$p$-$H$-Hölder} on $[a,b]$, that is,
	\begin{equation}
	|g^C(x)-g^C(y)| \leq H |x-y|^p,\; \forall x, \, y \in [a,b],
	\end{equation}
	where $H>0$ and $p \in (0,1]$. Then,
	\begin{equation}
	\begin{array}{c}
	\displaystyle
	\left| f(a)(g^C(b)-g^C(a)) +\sum_{s\in[a,b)}f(s)\Delta^+ g(s) -\int_{[a,b)} f\, \dif\mu_g 
	\right| \\
	\displaystyle
	\leq H (b-a)^p \Var_a^b f
	\end{array}
	\end{equation}
	and
	\begin{equation}
	\begin{array}{c}
	\displaystyle
	\left| \frac{f(a)+f(b)}{2}(g^C(b)-g^C(a))+\sum_{s\in[a,b)}f(s)\Delta^+ g(s) -\int_{[a,b)} f\, \dif\mu_g 
	\right| \\
	\displaystyle
	\leq H \left(\frac{b-a}{2}\right)^p \Var_a^b f.
	\end{array}
	\end{equation}
\end{lem}

\begin{rem} The previous quadrature formulae are most interesting in those cases where $D_g$ is finite, for it is under those circumstances that the sums involved become finite.
\end{rem}

\begin{proof} Thanks to Lemma~\ref{lemfi} it is enough to show that, if $g^C$ is
	$p$-$H$-Hölder on $[a,b]$ and $f\in BV([a,b])$, then
	\begin{equation} \label{eq:ec4a}
		\displaystyle
	\left| f(a)(g^C(b)-g^C(a)) -\int_a^b f(s) \, \dif g^C(s)
	\right| \leq \displaystyle H (b-a)^p \Var_a^b f, \end{equation}
	and
	\begin{equation}\label{eq:ec4b}
	\displaystyle 
	\left|
	\frac{f(a)+f(b)}{2}(g^C(b)-g^C(a))-\int_a^b f(s) \, \dif g^C(s)
	\right|\leq\displaystyle H \left(\frac{b-a}{2}\right)^p \Var_a^b f.
	\end{equation}
	Indeed, we can adapt the techniques in \cite{DRAGOMIR2011} for the Riemann-Stieltjes integral to the case of the Kurzweil-Stieltjes'. On one hand, by \cite[Theorem 6.3.6]{ANTUNES2019}, given $h:[a,b]\rightarrow \mathbb{R}$ continuous and 
	$f \in BV([a,b])$, it holds that
	\begin{equation} \label{eq:ec3}
	\left|\int_a^b h(s) \, \dif f(s)\right| \leq \|h\| \Var_a^b f.
	\end{equation}
	On the other, thanks to \cite[Theorem 6.4.2]{ANTUNES2019} (integration by parts),
	\begin{equation}
	\begin{array}{c}
	\displaystyle 
	\int_a^b \left[ g^C(s)-\frac{g^C(a)+g^C(b)}{2}\right] \, \dif f(s) =
	-\int_a^b f(s) \, \dif g^C(s) \\ 
	\displaystyle
	+\left. \left[g^C(x) -\frac{g^C(a)+g^C(b)}{2}\right] f(x) \right|_a^b 
	\\ 
	\displaystyle
	+\sum_{a\leq x \leq b} 
	\left(\Delta f^-(x) \Delta^- {g^C}(x)-\Delta^+ f(x) \Delta^+ {g^C}(x)\right),
	\end{array}
	\end{equation}
from where, given that $g^C$ is continuous,
	\begin{equation}
	\begin{array}{c}
	\displaystyle 
	\int_a^b \left[ g^C(s)-\frac{g^C(a)+g^C(b)}{2}\right] \, \dif f(s) =
	-\int_a^b f(s) \, \dif g^C(s) \\ 
	\displaystyle
	+\frac{f(a)+f(b)}{2}(g^C(b)-g^C(a)).
	\end{array}
	\end{equation}
In particular,
	\begin{equation}
	\begin{aligned}
	& \left|
	\frac{f(a)+f(b)}{2}(g^C(b)-g^C(a))-
	\int_a^b f(s) \, \dif g^C(s)\right| \\
	\leq & \left|
	\int_a^b \left[ g^C(s)-\frac{g^C(a)+g^C(b)}{2}\right] \, \dif f(s)
	\right|
	\end{aligned}
	\end{equation}
Let us define $h:t\in [a,b]\rightarrow h(t)$ as
	\begin{equation}
	h(t)=g^C(t)-\frac{g^C(a)+g^C(b)}{2}.
	\end{equation}
We have that
	\begin{equation}
	\begin{array}{rcl}
	\displaystyle |h(t)|&\leq&\displaystyle 
	\left| \frac{g^C(t)-g^C(a)+g^C(t)-g^C(b)}{2}
	\right| \\ 
	&\leq & \displaystyle 
	\frac{1}{2}|g^C(t)-g^C(a)| +\frac{1}{2}|g^C(b)-g^C(t)|
	 \\ 
	&\leq & \displaystyle 
	\frac{1}{2} H\left[(t-a)^p+(b-t)^p \right] \leq H \left(\frac{b-a}{2}\right)^p,
	\end{array}
	\end{equation}
for every $t \in [a,b]$. Thence, thanks to the bound (\ref{eq:ec3}), 
we obtain the bound (\ref{eq:ec4b}). In order to prove (\ref{eq:ec4a}) we can proceed in an analogous fashion integrating by parts:
	\begin{equation}
	\begin{array}{c}
	\displaystyle \int_a^b \left[g^C(s)-g^C(b)\right] \dif f(s)=-\int_a^b f(s) \, \dif g^C(s)
	\\ 
	\displaystyle
	+\left. \left[ g^C(x)-g^C(b) \right] f(x) \right|_a^b,
	\end{array}
	\end{equation}
	where we have already canceled out the terms concerning the sum. From the previous expression we obtain
	\begin{equation}
	\begin{aligned}
&
	\left| f(a)(g^C(b)-g^C(a))-\int_a^b f(s) \, \dif g^C(s)\right| \\
 \leq &\left| \int_a^b [g^C(s)-g^c(b)] \, \dif f(s) \right| \leq 
	H (b-a)^p \Var_a^b f.
\end{aligned}
	\end{equation}
\end{proof}

As we will see later on, it will be of special interest to consider the case when $D_f\subset D_g$ and $f^C$ behaves in a similar way to $g^C$. In such a case we can sharpen the previous quadrature formulae to obtain Lemma~\ref{lem4}.

\begin{dfn}Let $f:[a,b]\to\bR$ and $g:[a,b]\to\bR$ be left-continuous and increasing. We say $f$ is \emph{$g$-Lipschitz continuous} with \emph{Lipschitz constant} $H$ if $|f(t)-f(s)|\le H|g(t)-g(s)|$ for every $t,s\in[a,b]$.
	\end{dfn}
\begin{lem}\label{lempro}Let $g:[a,b]\to\bR$ be left-continuous and increasing $f:[a,b]\to\bR$ $g$-Lipschitz continuous. Then $f$ is $g$-continuous, bounded, $g$-integrable and of bounded variation.
	\end{lem}
\begin{proof}
	It is clear that $f$ is $g$-continuous. Since $g$ is bounded and $|f(t)-f(s)|\le H|g(t)-g(s)|$ for every $t,s\in[a,b]$, $f$ is bounded as well.

	 The $g$-integrability is an straightforward consequence of the definition of the Riemann-Stieltjes integral and the fact that $f$ is $g$-continuous. Finally, $\Var_a^b f\le H\Var_a^b g=H(g(b)-g(a))$, so $f$ is of bounded variation.
	\end{proof}

\begin{cor} Let $g:[a,b]\to\bR$ be left-continuous and 
increasing with $g^C$ being $p$-$H$-Hölder on $[a,b]$. Let $f:[a,b]\to\bR$ be 
$g$-Lipschitz continuous with Lipschitz constant $H$. Then,
	\begin{equation} \label{eq:ec4b1}
	\begin{array}{c}
	\displaystyle \left| f(a)(g^C(b)-g^C(a)) -\int_a^b f(s) \, \dif g^C(s)
	\right|\\ 
	\leq \displaystyle H^2 (b-a)^p(g(b)-g(a)) , 
	\end{array}
	\end{equation}
	\begin{equation}\label{eq:ec4b2}
	\begin{array}{c}
	\displaystyle 
	\left|
	\frac{f(a)+f(b)}{2}(g^C(b)-g^C(a))-\int_a^b f(s) \, \dif g^C(s)
	\right|\\ \displaystyle 
	\leq\displaystyle H^2 \left(\frac{b-a}{2}\right)^p (g(b)-g(a)).
	\end{array}
	\end{equation}
\end{cor}	

Error estimates obtained in the previous formula are not enough for our 
proposes, that is, proving the convergence of the numerical approximation to the solution 
of problem (\ref{eq:ec1}). In order to improve the previous estimations 
we must add some extra requirements to the continuous part of 
functions $g$ and $f$. 

In the next lemma and corollary we will prove that if $f$ is a g-Lipschitz 
continuous function, some properties of $g^C$ and $g^B$ are transferred to 
$f^C$ and $f^B$ respectively. In particular, we will see that is $f$ is a $g$-Lipschitz 
continuous function and $g^C$ is Lipschitz continuous then $f^C$ is also 
Lipschitz continuous. This property will be fundamental in order to improve the 
previous quadrature formula. 

For the next lemma we denote by $\fC(X)$ the set of connected components of $X\ss\bR$.
\begin{lem}
 Let $g:[a,b]\to\bR$ be  left-continuous in 
 $(a,b)$ and increasing and $f:[a,b]\to\bR$ be $g$-Lipschitz continuous with 
 Lipschitz constant $H$. Then $f^B$ is $g^B$-Lipschitz continuous with 
 Lipschitz constant $H$.
 Furthermore, if $f$ is increasing, then $f^C$ is $g^C$-Lipschitz continuous with  Lipschitz constant $H$.
\end{lem}

\begin{proof}
 Let $t\in[a,b)$, $s\in(x,b)$. Then
 \begin{equation}
 \begin{array}{rcl}
 \displaystyle
 |f^B(s)-f^B(t)|&\le& \displaystyle |f(s)-f(t)|+|f^C(s)-f^C(t)|\\
 &\le& \displaystyle H(g(s)-g(t))+|f^C(s)-f^C(t)|.
 \end{array}
 \end{equation}
 And thus, taking the limit when $s$ tends to $t$ from the right,
 \begin{equation}|\Delta^+f(t)|=|\Delta^+f^B(t)|\le H\Delta^+g(t)=H\Delta^+g^B(t).\end{equation}
 Hence, for $t\in[a,b)$, $s\in(x,b)$,
 \begin{equation}
 \begin{aligned}
 |f^B(t)-f^B(s)|= & \left|\sum_{r\in[t,s)}\Delta^+ f(r)\right|\le\sum_{r\in[t,s)}\left|\Delta^+ f(r)\right|\le\sum_{r\in[t,s)}H\Delta^+ g(r)\\= & H(g^B(s)-g^B(t))
 \end{aligned}
 \end{equation}
 Therefore, $f^B$ is $g^B$-Lipschitz and $g$-Lipschitz with constant $H$.

We know that $\sum_{t\in[a,b)}\Delta^+g(t)<\infty$. This implies, on one hand, that $D_g=\{t_n\}_{n\in\Lambda}$ with $\Lambda\subset \bN$ is countable and, on the other, that $\Delta^+g(t_{n})\to 0$. Observe that $g$ es continuous at $b$ and, either $g$ is continuous at $a$, or $a=t_k$ for some $k\in\Lambda$. In this last case we will assume, without loss of generality, that $a=t_1$.

Thus, consider, for $n\in\bN$, the functions 
\begin{equation}\begin{aligned}f_n(t):=f(t)-
\sum_{\substack{k \leq n\\k\in\Lambda}}\Delta^+ f(t_k) \chi_{(t_k,b]}(t)=f^C(t)+
\sum_{\substack{k > n\\k\in\Lambda}}\Delta^+ f(t_k) \chi_{(t_k,b]}(t)
,\\ g_n(t)=g(t)-
\sum_{\substack{k \leq n\\k\in\Lambda}}\Delta^+ g(t_k) \chi_{(t_k,b]}(t)=g^C(t)-
\sum_{\substack{k > n\\k\in\Lambda}}\Delta^+ g(t_k) \chi_{(t_k,b]}(t).\end{aligned}\end{equation}

Given $A\in\fC([a,b]\bs \{t_k\}_{k=1}^n)$ and $t,s\in A$, $t<s$, since there are 
no jumps of $\{t_k\}_{k=1}^n$ in $[t,s]$, $f(s)-f(t)=f_n(s)-f_n(t)$ and 
$g(s)-g(t)=g_n(s)-g_n(t)$, so $|f_n(s)-f_n(t)|\le H|g_n(s)-g_n(t)|$. Define, 
for $t,s\in[a,b]$, $t\le s$, 
\begin{align}F_n(t,s):= H(g_n(s)-g_n(t))-(f_n(s)-f_n(t)).
\end{align} 
Since $F_n(t,s)\ge 0$ for every $t,s\in A$, and $f_n$ and $g_n$ are continuous at 
the points of $\partial A$, it also holds for $t,s\in \overline A$. Furthermore, 
$F_n(t,x)+F_n(x,s)=F_n(t,s)$. Hence, if $F_n(t,s)\ge 0$ for 
$t,s\in [\a,\b]$ and $t,s\in[\b,\c]$ then $F_n(t,s)\ge 0$ for 
$t,s\in [\a,\c]$. To see this, just observe that if $t\in[\a,\b]$ 
and $s\in[\b,\c]$ we consider $F_n(t,s)=F_n(t,\b)+F_n(\b,s)\ge 0$.

Now, for any $t,s\in[a,b]$, $t<s$, either $[t,s]\cap \{t_k\}_{k=1}^n=\emptyset$ and 
thus $F_n(t,s)\ge 0$, or $[t,s]\cap \{t_k\}_{k=1}^n=\{t_k\}_{k=p}^q$, and
\begin{equation}F_n(t,s)=F_n(t,t_p)+F_n(t_p,t_{p+1})+\cdots+F_n(t_{q-1},t_{q})+F_n(t_q,t)\ge 0.
\end{equation}
We conclude that $F_n\ge0$.

Observe now that $f_n$ converges uniformly to $f^C$ and $g_n$ converges uniformly to $g^C$, so $F_n$ converges uniformly to 
\begin{align}F(t,s):= H(g^c(s)-g^C(t))-(f^C(s)-f^C(t)).
\end{align}  
Since $F_n\ge 0$ for every $n\in\bN$, $F\ge 0$ and thus,  $f^C$ is $g^C$-Lipschitz continuous with Lipschitz constant $H$.
\end{proof}

\begin{cor}\label{corld}
 Let $g:[a,b]\to\bR$ be left-continuous in 
 $(a,b)$ and increasing, $f:[a,b]\to\bR$ be $g$-Lipschitz continuous with 
 Lipschitz constant $H$. Then $f^C$ is $g^C$-Lipschitz continuous with 
 Lipschitz constant $H$.
 \end{cor}
 
\begin{proof} Since $f$ is $g$-Lipschitz continuous it is $g$-absolutely continuous and by  \cite[Theorem~5.4  (Fundamenta Theorem of Calculus)]{POUSO2015} there exists $f'_g$ $\mu_g$-a.e. and $f(t)=\int_{[a,t)}f'_g(s)\dif\mu_g(s)$. Since $f$ is $g$-Lipschitz continuous with Lipschitz constant $H$, for $s>t$,
	\begin{equation}\left|\int_{[t,s)}f'_g(r)\dif\mu_g(r)\right|=|f(s)-f(t)|\le H(g(s)-g(t)).\end{equation}
	Thus, by the definition of the Stieltjes $g$-derivative, $|f'_g|\le H$ $\mu_g$-a.e. 
	
	 Let $P=(f'_g)^{-1}(\bR^+)$ and $N=[a,b]\bs P$. And define \begin{equation}f_1(t)=\int_{[a,t)\cap P}f'_g(s)\dif\mu_g(s), \quad f_2(t)=-\int_{[a,t)\cap N}f'_g(s)\dif\mu_g(s).\end{equation} Clearly, both $f_1$ and $f_2$ are $g$-Lipschitz continuous with Lipschitz constant $H$ and increasing, so  $f_1^C$ and $f_2^C$ are $g^C$-Lipschitz with Lipschitz constant $H$. Thus, $f^C=f_1^C-f_2^C$ is $g$-Lipschitz with Lipschitz constant $H$.
\end{proof}

\begin{cor}  Let $g:[a,b]\to\bR$ be continuous in $\{a,b\}$, left-continuous in 
 $(a,b)$ and increasing, $f:[a,b]\to\bR$ be $g$-Lipschitz continuous with 
 Lipschitz constant $H$. If $g^C$ is Lipschitz continuous with Lipschitz constant 
$H$, then $f^C$ is Lipschitz continuous with Lipschitz constant 
$H^2$.
\end{cor}

\begin{proof} Is a direct consequence of previous corollary since for all $s,t \in [a,b]$
\begin{equation}
|f^C(s)-f^C(t)|\leq H |g^C(s)-g^C(t)|\leq H^2 |s-t|.
\end{equation}
\end{proof}

In order to simplify the notation, from now on we will assume, when necessary, 
that both the continuous part of $f$ and the continuous part of $g$ are Lipschitz 
continuous with the same Lipschitz constant $H$ (if necessary, we redefine $H$ to be $\max\{H,H^2\}$).

\begin{lem} \label{lem4} Let $g:[a,b]\to\bR$ be left-continuous and 
increasing and $f:[a,b]\to\bR$ be $g$-Lipschitz continuous with 
Lipschitz constant $H$. We also assume that $f^C$ and 
$g^C$ are Lipschitz continuous with Lipschitz constant 
$H$. Then,
\begin{equation} \label{eq:ec12}
\begin{array}{c}
\displaystyle
\bigg| f^C(a)(g^C(b)-g^C(a))\\
\displaystyle
+\sum_{s\in[a,b)}\left[f(s)\Delta^+ g(s)+
\Delta^+ f(s)(g^C(b)-g^C(s))\right]
- \int_{[a,b)} f \, \dif\mu_g \bigg| 
 \\ \displaystyle
= \bigg|
f^C(a)(g^C(b)-g^C(a))+f(a)\Delta^+ g(a)+
\Delta^+ f(a)(g^C(b)-g^C(a)) \\ \displaystyle 
 +\sum_{k\in\Lambda}\left[f(d_k)\Delta^+ g(d_k)+
\Delta^+ f(d_k)(g^C(b)-g^C(d_k))\right]-
\int_{[a,b)} f \, \dif\mu_g \bigg|\\ \leq H^2 (b-a)^2,
\end{array}
\end{equation}
and 
\begin{equation} \label{eq:ec13}
\begin{array}{c}
\displaystyle 
\bigg| 
\frac{f^C(a)+f^C(b)}{2}(g^C(b)-g^C(a))
 \\ \displaystyle
+\sum_{s\in[a,b)}\left[f(s)\Delta^+ g(s)+
\Delta^+ f(s)(g^C(b)-g^C(s))\right]-
\int_{[a,b)} f \, \dif\mu_g \bigg|
 \\ \displaystyle
 = \bigg|
\frac{f^C(a)+f^C(b)}{2}(g^C(b)-g^C(a))+f(a)\Delta^+ g(a)+
\Delta^+ f(a)(g^C(b)-g^C(a)) \\ \displaystyle 
+\sum_{k\in\Lambda}\left[f(d_k)\Delta^+ g(d_k)+
\Delta^+ f(d_k)(g^C(b)-g^C(d_k))\right]-
\int_{[a,b)} f \, \dif\mu_g \bigg| \\ \displaystyle
\leq \frac{H^2}{2} (b-a)^2,
\end{array}
\end{equation}
where $\{d_k\}_{k\in\Lambda}$ is the set of discontinuities
of	$g$ in $(a,b)$.
\end{lem}

\begin{proof}

First observe that, since $f$ is $g$-continuous, $D_f\subset D_g$ --see \cite[Proposition~3.2]{POUSO2017}.
Separating the jump part from the continuous part in both $f$ and $g$ we have that	
	\begin{equation}
	\begin{array}{rcl}
	\displaystyle
	\int_a^b f(s) \, \dif g(s)&=&
	\displaystyle
	\int_a^b f^B(s) \, \dif g^B(s)+
	\int_a^b f^B(s)\, \dif g^C(s)

	\\ && \displaystyle
	+\int_a^b f^C(s)\, \dif g^B(s)
	+\int_a^b f^C(s)\, \dif g^C(s),
	\end{array}
	\end{equation}
where the three first integrals correspond to series and the forth can be approximated using the previous quadrature formulae. Indeed, using analogous reasoning as that in \cite[Theorems~6.3.12 and 6.3.13]{ANTUNES2019}:
	\begin{equation}
	\begin{array}{rcl}
	\displaystyle \int_a^b f^B(s)\, \dif g^B(s)&=& 
	\displaystyle f^B(a) \Delta^+ g(a)+\sum_{k\in\Lambda}
	f^B(d_k) \Delta^+ g(d_k)
	 \\
	\displaystyle \int_a^b f^B(s)\, \dif g^C(s)&=&
	\displaystyle \Delta^+ f(a) (g^C(b)-g^C(a)) \\ && \displaystyle +
	\sum_{k\in\Lambda}\Delta^+ f(d_k) (g^C(b)-g^C(d_k)), 
	 \\
	\displaystyle
	\int_a^b f^C(s)\, \dif g^B(s)&=&
	\displaystyle f^C(a) \Delta^+ g(a)+\sum_{k\in\Lambda}
	f^C(d_k) \Delta^+ g(d_k).
	\end{array}
	\end{equation}
	Taking that into account,
	\begin{equation}
	\begin{array}{rcl}
	\displaystyle \int_a^b f(s) \, \dif g(s)&=&
	\displaystyle f(a) \Delta^+ g(a)+\sum_{k\in\Lambda}
	f(d_k) \Delta^+ g(d_k) \\
	&&\displaystyle +
	\Delta^+ f(a) (g^C(b)-g^C(a)) \\ && +\displaystyle
	\sum_{k\in\Lambda}\Delta^+ f(d_k) (g^C(b)-g^C(d_k))
	\\
	&&\displaystyle +
	\int_a^b f^C(s) \dif g^C(s).
	\end{array}
	\end{equation}
	Now, using the same argumentation as before,
	\begin{equation}
	\left|\frac{f^C(a)+f^C(b)}{2}(g^C(b)-g^C(a))-\int_a^b f^C(s) \, \dif g^C(s) \right|
	\leq \frac{H^2}{2} (b-a)^2,
	\end{equation}
	since $\Var_a^b f^C \leq H(b-a)$. The proof of identity~\eqref{eq:ec12} is analogous.
\end{proof}

\section{Description of the numerical method}\label{DNM}

In this section we present a predictor-corrector method based on the previous quadrature formulae. We will assume $g:[0,T]\to\bR$ is continuous at $t_0=0$ and that $D_g$ is finite. Let $x\in \mathcal{BC}_g([0,T])$ a solution of problem (\ref{eq:ec1}) and, in what follows, let $f_*(x)(t):=f(t,x(t))$,
\begin{equation}
\begin{aligned}
 f_*^B(x)(t):= & \sum_{s\in[a,t)}\Delta^+ f_*(x)(s),\\
 f_*^C(x)(t):= & f_*(x)(t)-f_*^B(x)(t).\\
\end{aligned}
\end{equation}
Consider now a set $\{t_k\}_{k=0}^{N+1}\subset [0,T]$ satisfying
\begin{itemize}
	\item[(H4)] $t_0=0$ and $t_{N+1}=b$; $t_{k+1}-t_k=h>0$, for every 
	$k=0,\ldots,N$ and $D_g \subset \{t_k\}_{k=0}^{N+1}$. We denote by 
	$K_1=\max\{\Delta^+g(d):\;d\in D_g\}$.
\end{itemize}
We also assume that
\begin{itemize}
\item[(H5)] $f_*(x)$ is $g$-Lipschitz continuous, in 
	particular, $D_{f_*(x)}\subset D_g$. $f_*^C(x)$ and $g^C$ are 
	Lipschitz-continuous with the same Lipschitz constant $H$.
\item[(H6)] $f(\cdot,c)\in \mathcal{BC}_g([0,T])$, for every $c\in \mathbb{R}$.
\end{itemize}
Let use define $x_k:=x(t_k)$ and $x_k^+:=x(t_k^+)$, for $k=0,\ldots,N+1$. By 
the definition of the Stieltjes derivative,
\begin{equation}
x_k^+=x_k+\Delta^+ g(t_k) f_*(x)(t_k),
\end{equation}
and, in the particular case $t_k \not\in D_g$, $\Delta^+g(t_k)=0$, so 
we have that $x_k^+=x_k$. Then, for every $k=0,\ldots,N+1$,
\begin{equation} \label{eq:ec5}
x_{k+1}=x_k+\int_{[t_k,t_{k+1})} f_*(x)(s)\, \dif g(s),
\end{equation}
where the integral is of Kurzweil-Stieltjes type. Using~\eqref{eq:ec12} on each interval,
\begin{equation} \label{eq:ec6}
\begin{array}{rcl}
\displaystyle \int_{[t_k,t_{k+1})} f_*(x)(s)\, \dif g(s) &\simeq & 
\displaystyle f_*^C(x)(t_k)(g^C(t_{k+1})-g^C(t_k))
 \\ && \displaystyle
+f_*(x)(t_k)\Delta^+ g(t_k)
 \\ && \displaystyle
+ \Delta^+ f_*(x)(t_k)(g^C(t_{k+1})-g^C(t_k)).
\end{array}
\end{equation}
Thus, in the case we use~\eqref{eq:ec13}, we have 
\begin{equation} \label{eq:ec7}
\begin{array}{rcl}
\displaystyle \int_{[t_k,t_{k+1})} f_*(x)(s)\, \dif g(s) & \simeq & 
\displaystyle \frac{f_*^C(x)(t_k)+f_*^C(x)(t_{k+1})}{2} (g^C(t_{k+1})-g^C(t_k))
 \\ && \displaystyle
+f_*(x)(t_k)\Delta^+ g(t_k)\\ && \displaystyle 
+ \Delta^+ f_*(x)(t_k)(g^C(t_{k+1})-g^C(t_k)).
\end{array}
\end{equation}
Observe that the condition $D_g \subset \{t_k\}_{k=0}^{N+1}$ implies that, on each interval, 
the quadrature formulae lose the terms related to the interior jumps. Restricting to $[t_k,t_{k+1}]$, we have
	\begin{equation}
	\begin{array}{rcl}
	\displaystyle 
	g^C(t)
	&=&
	\displaystyle g(t)-\chi_{(t_k,t_{k+1}]} \Delta^+ g(t_k),
	 \\
	\displaystyle 
	f_*^C(x)(t)
	&=&
	\displaystyle
	f_*(x)(t)-\chi_{(t_k,t_{k+1}]} \Delta^+ f_*(x)(t_k).
	\end{array}
	\end{equation}
Hence,
	\begin{equation}
	\begin{array}{rcl}
	\displaystyle \Delta^+ g(t_k)
	&=&
	\displaystyle
	g(t_k^+)-g(t_k),\\
	\displaystyle 
	g^C(t_k)
	&=&
	\displaystyle 
	g(t_k), \\
	\displaystyle g^C(t_{k+1})&=&\displaystyle
	g(t_{k+1})-g(t_k^+)+g(t_k),
	\end{array}
	\end{equation}
and
	\begin{equation}
	\begin{array}{rcl}	
	\displaystyle 
	\Delta^+ f_*(x)(t_k) &=&
	\displaystyle 
	f(t_k^+,x_k^+)-f(t_k,x_k), \\
	\displaystyle 
	f^C_*(x)(t_k)&=&\displaystyle f(t_k,x_k), \\
	\displaystyle 
	f^C_*(x)(t_{k+1})&=&\displaystyle
	f(t_{k+1},x_{k+1})-f(t_k^+,x_k^+)+f(t_k,x_k).
	\end{array}
	\end{equation}
Taking into account the previous formulae, it is transparent that, if we want to 
use~\eqref{eq:ec7} to approximate (\ref{eq:ec5}), the a-priori ignorance 
of the value $x_{k+1}$ forces us to estimate it. In order to do this we just 
have to use~\eqref{eq:ec6}, for which no further estimation is needed. 
That is, we will use~\eqref{eq:ec6} as predictor and~\eqref{eq:ec7} as corrector. 
Thus, the method will be as follows. Given 
$u_0=x_0$, we compute $\{(u_{k-1}^+,u_k^*,u_k)\}_{k=1}^{N+1}$ as
\begin{equation}\label{eq:ec11}
\left\{\begin{array}{rcl}
 u_k^+ &=&\displaystyle u_k+f(t_k,u_k)\Delta^+ g(t_k), \\
u_{k+1}^*&=&\displaystyle u_k^+ + f(t_k^+, u_k^+ )(g(t_{k+1})-g(t_k^+)), 
\\
u_{k+1} &=& \displaystyle u_k^+ + \frac{1}{2} \left(f(t_k^+, u_k^+ )+f(t_{k+1},
u_{k+1}^*)\right)(g(t_{k+1})-g(t_k^+)).
\end{array}\right.
\end{equation}

\section{Error analysis}\label{EA}

In this section we analyze the numerical method introduced in the previous one. As it happens with those numerical methods based on quadrature formulae --cf. \cite{ISAACSON1996,KINCAID1996}, it will be crucial at this point to study the error of approximating the integral in this way.

As before, we will need certain regularity hypotheses on the derivator $g$ as well 
as on the function $f$ when composed with the solution of the problem. Thus, 
we will assume the hypotheses (H1), (H2) and (H3) necessary  to guarantee the existence of 
problem (\ref{eq:ec1}) --see page \pageref{sh}; the hypotheses (H4), (H5) and (H6) 
established in the previous sections in order to formulate the numerical method and the following additional hypotheses for 
proving the convergence of the method: 
\begin{enumerate}
	\item[(H7)] $f(t,\cdot)\in\mathcal{C}^1(\mathbb{R})$ for every $t \in \mathbb{R}$ and 
	there exits $K_2\in\bR$ such that \[\left| \frac{\partial}{\partial x}f(t,x)\right|<K_2\]
	for every $(t,x) \in [0,T]\times\mathbb{R}$. 
	\item[(H8)] $f(t^+,\cdot) \in 
	\mathcal{C}^1(\mathbb{R})$ for every $t \in \mathbb{R}$ and 
	there exits $K_3\in\bR$ such that \[\left| \frac{\partial}{\partial x}f(t^+,x)\right|<K_3\]
	for every $(t,x) \in [0,T]\times\mathbb{R}$. 	
\end{enumerate}
We must emphasize that the above hypotheses are not independent. For example, hypothesis (H6) implies (H1)-(H2) and 
(H7) implies (H3). Therefore, for our purposes it is sufficient that the hypotheses (H4)-(H8) are fulfilled. We will now 
establish the basic notions related to the truncating error associated to the quadrature formula of the predictor-corrector method.

\begin{dfn}[Local truncating error associated to the quadrature formula] Given a partition $P$ satisfying (H4), we define:
	\begin{itemize}
		\item The local error associated to the formula (\ref{eq:ec12}) as
		\begin{equation}
	\sigma_{k+1}^*:=	x_{k+1}-x_k^+ -f(t_k^+,x_k^+ ) (g(t_{k+1})-g(t_k^+)),
		\end{equation}
		with $k=0,\ldots,N$.
		\item The local error associated to the formula
		 (\ref{eq:ec13}) as
		\begin{equation}
			\sigma_{k+1}:=x_{k+1}-x_k^+ - \frac{1}{2}\left(f(t_k^+,x_k^+ )+f(t_{k+1},x_{k+1})\right) 
		(g(t_{k+1})-g(t_k^+)),
		\end{equation}
		with $k=0,\ldots,N$.
		\item Letting \[x_{k+1}^*:=x_{k+1}-\sigma_{k+1}^*=
		x_k^+ +f(t_k^+,x_k^+ ) (g(t_{k+1})-g(t_k^+)),\] we define the local truncating error of the predictor-corrector method associated to (\ref{eq:ec11}), in terms of the exact solution, as
		\begin{equation}
		\begin{array}{c}
		\displaystyle 
	\tau_{k+1}:=	x_{k+1}-x_k^+ - \frac{1}{2}\left(f(t_k^+,x_k^+ )+f(t_{k+1},x_{k+1}^*)\right) 
		(g(t_{k+1})-g(t_k^+)),
		\end{array}
		\end{equation}
		with $k=0,\ldots,N$.
	\end{itemize}
\end{dfn}

\begin{rem}It is usual to find in the literature local truncating errors defined relative to the discretization step $h$, that is, $\widetilde{\sigma}_{k+1}^*=\sigma_{k+1}^*/h$, 
	$\widetilde{\sigma}_{k+1}=\sigma_{k+1}/h$ and $\widetilde{\tau}_{k+1}=
	\tau_{k+1}/h$ --cf. \cite{ISAACSON1996,KINCAID1996}. In those cases, for $k=0,\ldots,N$,
	\begin{equation}
	\begin{aligned}
	x_{k+1}= &
	x_k^+ +f(t_k^+,x_k^+ ) (g(t_{k+1})-g(t_k^+))+h \widetilde{\sigma}_{k+1}^*,
	 \\
	x_{k+1}=& x_k^+ + \frac{1}{2}\left(f(t_k^+,x_k^+ )+
	f(t_{k+1},x_{k+1})\right) 
	(g(t_{k+1})-g(t_k^+)) + h\widetilde{\sigma}_{k+1}, \\
	x_{k+1}=&  x_k^+ + \frac{1}{2}\left(f(t_k^+,x_k^+ )+
	f(t_{k+1},x_{k+1}^*)\right) (g(t_{k+1})-g(t_k^+))+h\widetilde{\tau}_{k+1}.
	\end{aligned}
	\end{equation}
	We have opted for the first set of errors in order to simplify the notation. In any case, the relation between both definitions is clear.
\end{rem}

Now we present some bounds of the errors.

\begin{lem} \label{lem2}
	For every $k=0,\ldots,N$ we have the following bounds:
	\begin{itemize}
		\item $\displaystyle |\sigma_{k+1}^*| \leq H^2 h^{2}$.
		\item $\displaystyle |\sigma_{k+1}|\leq \frac{1}{2}H^2 h^{2}$.
		\item $\displaystyle |\tau_{k+1}| \leq \frac{1}{2} H^2 h^2 + \frac{1}{2} K_2 H^3 h^3$.
	\end{itemize}

\end{lem}

\begin{proof} The two first assertions are a direct consequence of Lemma~\ref{lem4}. In order to obtain the third one, we manipulate the definition of $\tau_{k+1}$ leaving
	\begin{equation}
	\begin{aligned}
	\displaystyle x_{k+1}=& \displaystyle x_k^+ +
	\frac{1}{2}\left( 
	f(t_k^+,x_k^+ )+f(t_{k+1},x_{k+1}^*)
	\right) (g(t_{k+1}-g(t_k^+))+\tau_{k+1} \\
=&\displaystyle
	x_k^+ + \frac{1}{2}\left(f(t_k^+,x_k^+ )+f(t_{k+1},x_{k+1})\right) 
	(g(t_{k+1})-g(t_k^+)) \\
	& \displaystyle + 
	\frac{1}{2}\left(f(t_{k+1},x_{k+1}^*)-f(t_{k+1},x_{k+1})\right) 
	(g(t_{k+1})-g(t_k^+))+\tau_{k+1},
	\end{aligned}
	\end{equation}
wherefrom we obtain, using the definition of $\sigma_{k+1}$,
	\begin{equation}
	\sigma_{k+1}=\frac{1}{2}\left(f(t_{k+1},x_{k+1}^*)-f(t_{k+1},x_{k+1})\right) 
	(g(t_{k+1})-g(t_k^+))+\tau_{k+1}.
	\end{equation}
	By the Mean Value Theorem of Differential Calculus, there exists $c_{k+1}$ in the open interval of extremities 
	$x_{k+1}$, $x_{k+1}^*$ such that
	\begin{equation}
	\sigma_{k+1}+\frac{1}{2}\frac{\partial f}{\partial x}(t_{k+1},c_{k+1}) 
	\sigma_{k+1}^*(g(t_{k+1})-g(t_k^+))=\tau_{k+1},
	\end{equation}
	thence, taking the absolute value,
	\begin{equation}\label{eqrefd}
	|\tau_{k+1}| \leq |\sigma_{k+1}|+\frac{1}{2} K_2 H|\sigma_{k+1}^*| h.
	\end{equation}
Using the bounds obtained for $|\sigma_{k+1}|$ and $|\sigma_{k+1}^*|$,
	\begin{equation}
	|\tau_{k+1}| \leq \frac{1}{2} H^2 h^2 + \frac{1}{2} K_2 H^3 h^3.
	\end{equation}
\end{proof}

\begin{cor}[Consistence of the numerical method] In the functional framework in which Lemma~\ref{lem2} is valid the method is consistent.
\end{cor}

\begin{proof} Indeed, thanks to the bounds provided by Lemma~\ref{lem2}, we obtain
	\begin{equation}
	0\leq \lim_{h \to 0} \max_{0\leq n \leq N} |\widetilde{\tau}_{n+1}|\leq 
	\lim_{h \to 0} 
	\frac{1}{2} H^2 h + \frac{1}{2} K_2 H^3 h^2=0,
	\end{equation}
	wherefrom we deduce the consistency of the method in the classical sense.
\end{proof}

\begin{rem}In view of the bounds in Lemma~\ref{lem2}  we must observe that the introduction of a predictor in the quadrature formula does not penalize its convergence order. This is due to the fact that  $\sigma_{k+1}^*$ (which is the predictor term in the formula) appears multiplied by $h$ in~\eqref{eqrefd}.

In our case, due to the regularity of the terms involved, we are not capable of improving the order of convergence of the two-point formula with respect to the one-point one. This is not usually the case, as in the literature we can see examples 
--for instance \cite{ISAACSON1996,KINCAID1996}-- where the two-point quadrature formula has a better convergence 
order --without the predictor penalizing the global order of the method-- than the one-point one.	

\end{rem}

\begin{dfn}[Local error of the algorithm]
	We define the following errors associated to the numerical algorithm.
	\begin{itemize}
		\item $e_k^+ := u_k^+ -x_k^+ $, with $k=0,\ldots,N$, is the local error of the corrector 
		regarding the limit from the right at $t_k$. It is clear that it does not make sense to consider 
		this error for $k=N+1$.
		\item $e_{k}^*=u_{k}^*-x_{k}^*$, where $k=1,\ldots,N+1$, is the local error of the predictor at 
		the point $t_k$.
		\item $e_{k}=u_{k}-x_{k}$, with $k=0,\ldots,N+1$ the local error of the predictor at the point $t_k$ and 
		$e_0$ is the error associated to the initial condition.
	\end{itemize}
\end{dfn}

In Lemma~\ref{lem3} we obtain bounds for the previous lemmata based on recurrence formulae which, afterwards, we will analyze in order to obtain bounds of the error at each of the points of the temporal discretization.

\begin{lem} \label{lem3} Under the hypotheses of Lemma~\ref{lem2}, we derive the following formulae for $e_{k+1}^*$, $e_{k+1}$ and 
$e_k^+$, with $k=0,\ldots,N$,
	\begin{itemize}
		\item $\displaystyle |e_k^+ | \leq \left[1+ K_1K_2 \chi_{D_g} (t_k)\right] |e_k|$.
		\item $\displaystyle |e_{k+1}^*| \leq \left[1+K_3 H h+(1+K_3 H h)K_1K_2 \chi_{D_g}(t_k) \right]|e_k|$.
		\item $\displaystyle |e_{k+1}| \leq \left[1+G_1+G_2 \chi_{D_g}(t_k) \right] 
		|e_k|+\tau_{k+1}$, where:
		\begin{equation}
		\begin{aligned}
		G_1=& \displaystyle 
		\frac{1}{2} K_2 H h +\frac{1}{2} K_3 H h + \frac{1}{2} K_3^2 H^2 h^2, 
		 \\ 
		G_2=& \displaystyle K_1 K_2 + 
		\frac{1}{2} K_1 K_2 K_3 H h+\frac{1}{2}K_1 K_2^2 H h + 
		\frac{1}{2}K_1 K_2^2K_3 H^2h^2.
		\end{aligned}
		\end{equation}
	\end{itemize}
\end{lem}

\begin{proof} We compute each of the error bounds separately.
	\begin{itemize}
		\item \emph{Local error of the corrector regarding the limit from the right}. We have that
		\begin{equation}
		\begin{aligned}
		e_k^+ 
		=& \displaystyle 
		u_k-x_k \\
		=& \displaystyle
		u_k+f(t_k,u_k) \Delta^+ g(t_k)-x_k-f(t_k,x_k)\Delta^+ g(t_k)
		 \\
		=& \displaystyle
		e_k+\left(f(t_k,u_k)-f(t_k,x_k) \right) \Delta^+ g(t_k)
		 \\
		=& \displaystyle
		e_k + \frac{\partial f}{\partial x} (t_k,c_k) e_k \Delta^+ g(t_k),
		\end{aligned}
		\end{equation}
		where $c_k$ belongs to the open interval of extremities
		$u_k$ and $x_k$. Taking the absolute value on both sides,
		\begin{equation}
		\begin{aligned}
		|e_k^+ | \leq & \displaystyle |e_k|+K_1K_2 |e_k| \chi_{D_k}(t_k), 
		 \\
		=& \displaystyle \left[1+K_1K_2 \chi_{D_g}(t_k)\right] |e_k|.
		\end{aligned}
		\end{equation}
		\item \emph{Local error of the predictor at the point}. We have
		\begin{equation}
		\begin{aligned}
		e_{k+1}^*=&\displaystyle u_{k+1}^*-x_{k+1} \\
		=&\displaystyle
		 u_k^+ + f(t_k^+, u_k^+ )(g(t_{k+1})-g(t_k^+))-x_k^+ \\ & \displaystyle-
		f(t_k^+,x_k^+ )(g(t_{k+1})-g(t_k^+)) \\
		=&\displaystyle e_k^+ + \left(f(t_k^+, u_k^+ )-f(t_k^+,x_k^+ ) \right)
		(g(t_{k+1})-g(t_k^+)) \\
		=&\displaystyle e_k^+ + \frac{\partial f}{\partial x}(t_k^+,c_k^+) e_k^+ 
		(g(t_{k+1})-g(t_k^+)),
		\end{aligned}
		\end{equation}
		where $c_k^+$ belongs to the open interval of extremities
		$ u_k^+ $ and $x_k^+ $. Taking the absolute value on both sides,
		\begin{equation}
		\begin{aligned}
		|e_{k+1}^*|\leq & \displaystyle \left[1+K_3 H h\right] |e_k^+ |
		 \\
		\leq& \displaystyle 
		\left[1+K_3 H h\right] \left[1+K_1K_2 \chi_{D_g}(t_k)\right] |e_k|
		 \\
		=& \displaystyle 
		\left[1+K_3 H h+(1+K_3 H h)K_1K_2 \chi_{D_g}(t_k) \right]|e_k|.
		\end{aligned}
		\end{equation}
		\item \emph{Local error of the predictor at the point}. We have that
		\begin{equation}
		\begin{aligned}
		|e_{k+1}|  \leq & \displaystyle 
		 u_k^+ + \frac{1}{2} (f(t_k^+, u_k^+ )+f(t_{k+1},
		u_{k+1}^*))(g(t_{k+1})-g(t_k^+)) \\
		& \displaystyle 
		-x_k^+ -
		\frac{1}{2}\left( 
		f(t_k^+,x_k^+ )+f(t_{k+1},x_{k+1}^*)
		\right) (g(t_{k+1}-g(t_k^+))-\tau_{k+1} \\
		=&\displaystyle |e_k^+ |+\frac{1}{2}\left(f(t_k^+, u_k^+ )-f(t_k^+,x_k^+ )\right)
		(g(t_{k+1})-g(t_k^+)) 
		 \\
		& + \displaystyle \frac{1}{2} \left(
		f(t_{k+1},u_{k+1}^*)-f(t_{k+1},x_{k+1}^*)
		\right)(g(t_{k+1})-g(t_k^+)) -\tau_{k+1} \\
		=&\displaystyle |e_k^+ | + \frac{1}{2} \frac{\partial f}{\partial x}(t_k^+,c_k^+)
		e_k^+ (g(t_{k+1})-g(t_k^+)) \\
		& \displaystyle + \frac{1}{2} \frac{\partial f}{\partial x}(t_{k+1},c_{k+1}^*)
		e_{k+1}^* (g(t_{k+1})-g(t_k^+)) -\tau_{k+1},
		\end{aligned}
		\end{equation}
		where $c_k^+$ belongs to the open interval of extremities 
		$ u_k^+ $, $x_k^+ $ and $c_{k+1}^*$ belongs to the open interval of extremities $u_{k+1}^*$ and $x_{k+1}^*$. Taking the absolute value on both sides,
		\begin{equation}
		\begin{aligned}
		|e_{k+1}|  \leq & \displaystyle \left[1+ \frac{1}{2} K_3 H h \right] |e_k^+ |+
		\frac{1}{2} K_2 H h |e_{k+1}^*| +\tau_{k+1} \\
		 \leq & \displaystyle
		\left[1+ \frac{1}{2} K_3 H h \right] \left[1+K_1K_2 \chi_{D_g}(t_k)\right] |e_k| 
		 \\
		& \displaystyle 
		+ \frac{1}{2} K_2 H h \left[1+K_3 H h+(1+K_3 H h)K_1K_2 \chi_{D_g}(t_k)
		\right] |e_k| \\ & +
		\tau_{k+1} \\ 
		=& \displaystyle
		\left[1+G_1+G_2 \chi_{D_g}(t_k) \right] |e_k|+\tau_{k+1},
		\end{aligned}
		\end{equation}
		where
		\begin{equation}
		\begin{aligned}
		G_1=& \displaystyle 
		\frac{1}{2} K_2 H h +\frac{1}{2} K_3 H h + \frac{1}{2} K_3^2 H^2 h^2, 
		 \\ 
		G_2=& \displaystyle K_1 K_2 + 
		\frac{1}{2} K_1 K_2 K_3 H h+\frac{1}{2}K_1 K_2^2 H h + 
		\frac{1}{2}K_1 K_2^2K_3 H^2h^2.
		\end{aligned}
		\end{equation}
	\end{itemize}
\end{proof}

\begin{rem} Observe that previous error formulae can be simplified in the case $t_k \notin D_g$. In this situation, those errors concerning the limit from the right coincide with the ones of the corrector at the point and we recover the classical error formulae.
\end{rem}

From the formulae in Lemma~\ref{lem3} we can prove the following result concerning the error of the numerical method.

\begin{lem} \label{lem4b} Under the hypotheses of Lemma~\ref{lem3}, we have, for $n=0,\ldots,N$,
	\begin{equation}
	|e_{n+1}|\leq \left[1+G_2 \right]^{\#D_g} \left[|e_0|+
	\frac{\tau}{G_1}\right] \exp\left[(n+1) G_1 \right].
	\end{equation}
	where $\tau=\max\{|\tau_k|:\; k=1,\dots,N+1\}$.
\end{lem}

\begin{proof} Using the notation of Lemma~\ref{lem3}, we have that
	\begin{equation}
	|e_{n+1}| \leq \left[1+G_1+G_2 \chi_{D_g}(t_n) \right] |e_n|+\tau.
	\end{equation}
 Thus, applying the previous bound recursively,
	\begin{equation}
	\begin{aligned}
	\displaystyle
	|e_{n+1}| \leq & \displaystyle
	\prod_{k=0}^n \left[1+G_1+G_2 \chi_{D_g}(t_k) \right] |e_0| \\ & +
	\sum_{k=1}^n \prod_{j=k}^n\left[1+G_1+G_2 \chi_{D_g}(t_j) \right]\tau+\tau 
	 \\
=&\displaystyle \prod_{k=0}^n \left[1+G_2 \chi_{D_g}(t_k) \right]
	\left[1+\frac{G_1}{1+G_2 \chi_{D_g}(t_k)} \right] |e_0| \\
	& \displaystyle + 
	\sum_{k=1}^n \prod_{j=k}^n\left[1+G_2 \chi_{D_g}(t_j) \right]
	\left[1+\frac{G_1}{1+G_2 \chi_{D_g}(t_j)} \right] \tau + \tau.
	\end{aligned}
	\end{equation}
	Accounting for the number discontinuities of the derivator (which we denote $\# D_g$) we obtain
	\begin{equation}
	\displaystyle |e_{n+1}| \leq 
	\displaystyle 
	\left[1+G_2 \right]^{\#D_g} \left[1+G_1 \right]^{n+1} 
	+ \left[1+G_2 \right]^{\#D_g} \sum_{k=0}^n 
	\left[1+G_1\right]^k \tau .
	\end{equation}
 Now, taking into account that, for a given number $G\geq 0$, 
	\begin{equation}
	1+\sum_{k=0}^n G(1+G)^k=(1+G)^{n+1}
	\end{equation}
	and that $1+G \leq \exp(G)$, we have
	\begin{equation}
	\begin{aligned}
	|e_{n+1}| \leq & \displaystyle 
	\left[1+G_2 \right]^{\#D_g} \left[1+G_1 \right]^{n+1} +
	\left[1+G_2 \right]^{\#D_g} \frac{\tau}{G_1} \sum_{k=0}^n 
	G_1(1+G_1)^k \\
	 \leq & \displaystyle 
	\left[1+G_2 \right]^{\#D_g} \left[|e_0|+
	\frac{\tau}{G_1}\right] \left[1+G_1 \right]^{n+1} 
	 \\
	 \leq & \displaystyle 
	\left[1+G_2 \right]^{\#D_g} \left[|e_0|+
	\frac{\tau}{G_1}\right] \exp((n+1) G_1).
	\end{aligned}
	\end{equation}
\end{proof}

Now we will prove the main theorem of this section. In it we will see that, in the framework of the previous results, we can guarantee the convergence of the method introduced in the previous section.

\begin{thm}[Convergence of the predictor-corrector method]\label{thmcon} Under the hypotheses of Lemma~\ref{lem4}, if we assume $e_0=0$, we have, for a given $t_j\in [0,T]$, that
	\begin{equation}
	\begin{array}{lcl}
	\displaystyle \lim_{h \to 0} |u_j-x(t_j)| &=& 0, \\
	\displaystyle \lim_{h \to 0} |u_j^*-x^*(t_j)| &=& 0, \\
	\displaystyle \lim_{h \to 0} | u_j^+-x(t_j^+)| &=& 0.
	\end{array}
	\end{equation}
	Furthermore, we get the following error bounds:
	\begin{equation}
	\begin{aligned}
	\displaystyle
	|u_j^*-x^*(t_j)|  \leq & \displaystyle
	\left[1+G_5 \chi_{D_g}(t_j)\right]
	\left[1+G_2 \right]^{\#D_g}
\\ 	 & \cdot \left[|e_0|+ h
	\frac{\tau/h}{G_1}\right] \exp\left[ \frac{G_1}{h} t_j + G_4 \right],
	 \\
	\displaystyle
	|u_j^+-x(t_j^+)| \leq & \displaystyle \left[1+G_3\chi_{D_g}(t_j) \right] 
	\left[1+G_2 \right]^{\#D_g} \\ 	 & \cdot \left[|e_0|+ h
	\frac{\tau/h}{G_1}\right] \exp\left[ \frac{G_1}{h} t_j \right],
	 \\
	\displaystyle
	|u_j-x(t_j)|  \leq & \displaystyle 
	\left[1+G_2 \right]^{\#D_g} \left[|e_0|+ h
	\frac{\tau/h}{G_1}\right] \exp\left[ \frac{G_1}{h} t_j \right],
	\end{aligned}
	\end{equation}
	for every $j=1,\ldots,N+1$, where
	\begin{equation}
	\begin{array}{rcl}
	G_3&=& \displaystyle K_1 K_2, \\
	G_4&=&\displaystyle K_3 H h , \\
	G_5&=&\displaystyle G_3(1+G_4).
	\end{array}
	\end{equation}
\end{thm}

\begin{proof} We analyze each case separately.
	\begin{itemize}
		\item \emph{Errors associated to the corrector}. From the previous lemma
		\begin{equation}
		\begin{array}{rcl}
		|e_{n+1}| & \leq & \displaystyle 
		\left[1+G_2 \right]^{\#D_g} \left[|e_0|+ h
		\frac{\tau/h}{G_1}\right] \exp\left[ \frac{G_1}{h} (n+1)h \right],
		\end{array}
		\end{equation}
	where
		\begin{equation}
		\begin{aligned}
		\displaystyle
		\frac{\tau/h}{G_1}= & \frac{H^2 + K_2 H^3 
			h}{ K_2 H + K_3 H + 
			K_3^2 H^2 h}\longrightarrow 
		\frac{H^2 }{K_2 H +K_3 H}>0, \\
		\displaystyle
		\frac{G_1}{h}= & 
		\frac{1}{2} K_2 H +\frac{1}{2} K_3 H + \frac{1}{2} K_3^2 H^2 h 
		\longrightarrow \frac{1}{2} K_2 H +\frac{1}{2} K_3 H>0,
		\end{aligned}
		\end{equation}
		when $h\to 0$. Hence, given $t_j \in [0,T]$, we get
		\begin{equation}
		|u_j-x(t_j)| \leq \displaystyle 
		\left[1+G_2 \right]^{\#D_g} \left[|e_0|+ h
		\frac{\tau/h}{G_1}\right] \exp\left[ \frac{G_1}{h} t_j \right],
		\end{equation}
 thence, given that $|e_0|=0$, we have the convergence of the corrector to the solution of the problem:
		\begin{equation}
		\lim_{h \to 0} u_j= x(t_j).
		\end{equation}
		\item \emph{Errors associated to the predictor}. 
 Using the bounds in Lemma~\ref{lem3} and Lemma~\ref{lem4} we have that, given $t_j \in [0,T]$,
		\begin{equation}
		\begin{array}{rcl}
		|u_j^*-x^*(t_j)| &\leq & \displaystyle
		\left[1+K_3 H h+(1+K_3 H h)K_1K_2 \chi_{D_g}(t_j)\right] |e_j| \\
		&\leq& \displaystyle 
		\left[1+G_5 \chi_{D_g}(t_j)\right]
		\exp (G_4)|e_j| 
		 \\
		&\leq& \displaystyle 
		\left[1+G_5 \chi_{D_g}(t_j)\right]
		\left[1+G_2 \right]^{\#D_g} \\ & & \cdot \left[|e_0|+ h
		\frac{\tau/h}{G_1}\right] \exp\left[ \frac{G_1}{h} t_j + G_4 \right],
		\end{array}
		\end{equation}
		from where we obtain the convergence.

		\item \emph{Errors associated to the right limit}. Using the bounds in Lemma~\ref{lem3} and Lemma~\ref{lem4} we have that, given $t_j \in [0,T]$,
		\begin{equation}
		|u_j^+-x(t_j^+)| \leq \displaystyle \left[1+G_3\chi_{D_g}(t_j) \right] 
		\left[1+G_2 \right]^{\#D_g} \left[|e_0|+ h
		\frac{\tau/h}{G_1}\right] \exp\left[ \frac{G_1}{h} t_j \right].
		\end{equation}
 We obtain the same kind of convergence as in the previous case. It is worth noting that, in the case $t_j \notin D_g$, the error associated to the right limit coincides with the error of the predictor.
	\end{itemize}
\end{proof}

\begin{rem}
 Observe that the order of convergence of the method equals the order of $\tau$ minus one, that is, of the order of $\widetilde{\tau}$. In the case we deal with functions with extra regularity we may be able to improve the order of $\tau$, which would better the order of convergence of the numerical method. Last, we would like to mention that the method we presented generalizes the classical order two Runge-Kutta. This assertion is motivated by the fact that the usual derivative is a particular instance of the Stieltjes derivative in the case $g(x)=x$. 
\end{rem}

Last, we analyze the stability of the method with the intention of evaluating its sensitivity towards the perturbations generated by the rounding errors produced while evaluating the different elements of scheme (\ref{eq:ec11}). 
We omit the proof of the following result, for it is essentially a modification of that of Theorem~\ref{thmcon}.

\begin{thm}[Stability of the numerical method] Given
	$\widehat{u}_0$, we consider the following modification of the numerical scheme (\ref{eq:ec11}):
	\begin{equation}
	\left\{\begin{array}{rcl}
	\displaystyle 
	\widehat{u}^+_k&=&
	\displaystyle
	\widehat{u}_k+f(t_k,\widehat{u}_k) 
	\Delta^+ g(t_k)+ \widehat\rho_k^+, \\
	\displaystyle 
	\widehat{u}_{k+1}^*&=& 
	\displaystyle \widehat{u}_k^+ + f(t_k,\widehat{u}^+_k) 
	(g(t_{k+1})-g(t_k^+)) + \widehat\rho_{k+1}^*, 
	\\
	\displaystyle \widehat{u}_{k+1} &=& 
	\displaystyle
	\widehat{u}_k^+ + 
	\frac{1}{2} \left(f(t_k,\widehat{u}_k^+)+f(t_{k+1},
	\widehat{u}_{k+1}^*)\right)(g(t_{k+1})-g(t_k^+)) + \widehat\rho_{k+1},
	\displaystyle 
	\end{array}\right.
	\end{equation}
	where $k=0,\ldots,N$. Defining $\widehat{e}_{k}=
	\widehat{u}_k-x(t_k)$, for $k=1,\ldots,N+1$, it holds that
	\begin{equation}
	|\widehat{e}_{k+1}| \leq \left[1+G_1+G_2 \chi_{D_g}(t_n) \right] |\widehat{e}_n|+
	|\tau_{k+1}|+|\widehat{\rho}_{k+1}|+ G_1 |\widehat{\rho}_k^+|+
	G_6 |\widehat{\rho}_{k+1}^*|,
	\end{equation}
	where
	\begin{equation}
	G_6=\frac{1}{2} K_2 H h.
	\end{equation}
	Thence, writing $\widehat{\rho}:=\max\{|\widehat{\rho}_k|: \;
	k=1,\ldots,N+1\}$, $\widehat{\rho}^*=\max\{|\widehat{\rho}_k^*|: \;
	k=1,\ldots,N+1\}$ and $\widehat{\rho}^+=\max\{|\widehat{\rho}_k^+|: \;
	k=0,\ldots,N\}$, we have that, for every $t_j \in [0,T]$, 
	\begin{equation}
	|\widehat{u}_j-x(t_j)| \leq \displaystyle 
	\left[1+G_2 \right]^{\#D_g} \left[|e_0|+
	h \frac{\frac{\tau}{h}+\frac{\widehat{\rho}}{h}+G_1 \frac{\widehat{\rho}^+}{h}+ 
		G_6 \frac{\widehat{\rho}^*}{h}}{G_1}\right] \exp\left[ \frac{G_1}{h} t_j \right].
	\end{equation}
\end{thm}

\section{The general linear equation}\label{GLE}

In order to validate the numerical approximation of the solution of problem~\eqref{eq:ec1}, we will 
consider the following general linear equation as a test problem:
\begin{equation}\label{eqlg}\begin{aligned}
x'_g(t)+d(t)x(t)= & h(t), t\in[0,T),\\
x(0)= & x_0,
\end{aligned}
\end{equation}
where $x_0\in{\mathbb R}$, $h,d\in L_g^1([0,T))$ and
\begin{align}
& \label{eql1} d(t)\Delta^+g(t)\ne 1\enskip\forall t\in[a,b)\cap D_g,\\
&\label{eql2} \sum_{t\in[a,b)\cap D_g}\left|\ln\left|1-d(t)\Delta^+g(t)\right|\right|<\infty.
\end{align}
Under~\eqref{eql1}--\eqref{eql2}, we know there is a unique solution of~\eqref{eqlg} which can be computed explicitly --see \cite{POUSO2017}-- as the unique solution of the problem 
\begin{equation}\label{eqlg2}\begin{aligned}
x'_g(t)+\widetilde d(t)x(t^+)= & \widetilde h(t), t\in[0,T),\\
x(0)= & x_0,
\end{aligned}
\end{equation}
where
\begin{align}
\label{eqd} \widetilde d(t):= & \frac{d(t)}{1-d(t)\Delta^+g(t)},\\
\label{eqh} \widetilde h(t):= & \frac{h(t)}{1-d(t)\Delta^+g(t)},
\end{align}
are $L_g^1([0,T))$ functions thanks to \cite[Proposition~6.8]{POUSO2017}. Therefore, by \cite[Proposition~6.7]{POUSO2017}, the solution of problem~\eqref{eqlg2} is given by
\begin{equation} \label{eq:sollingen}
x(t):=\widehat e_{\widetilde d}^{-1}(t)\left(x_0+\int_{[0,t)} \widehat 
e_{\widetilde d}(s)\widetilde h(s)\operatorname{d} \mu_g(s)\right),
\end{equation}
where, given an element $c\in L_g^1([0,T))$, 
\begin{equation}
\widehat e_{c}(t):=\begin{dcases} \exp\left(\int_{[0,t)}\widehat c(s)\operatorname{d}\mu_g(s)\right), & t\in[0,s_1],\\ (-1)^k\exp\left(\int_{[0,t)}\widehat c(s)\operatorname{d}\mu_g(s)\right), & t\in[s_k,s_{k+1}],\ k=1,\dots,N,
\end{dcases}
\end{equation}
\begin{equation}
\widehat c(t):=\begin{dcases} c(t) & t\in[0,T]\setminus D_g,\\ \frac{\ln|1+c(t)\Delta^+ g(t)|}{\Delta^+g(t)}, & t\in[0,T)\cap D_g,
\end{dcases}
\end{equation}
being $\{s_1,\dots,s_N\}=\{t\in[0,T)\cap D_g\ :\ 1+c(t)\Delta^+ g(t)<0\}$ the set of points such that 
$1+c(t)\Delta^+ g(t)<0$ and $s_{N+1}=T$. This set has finite cardinality --see \cite[ Lemma~6.4]{POUSO2017}. 
In our case, $c=\widetilde d$, thus:
\begin{equation}
1+\widetilde d(t)\Delta^+ g(t)=1+\frac{d(t)\Delta^+g(t)}{1-d(t)\Delta^+g(t)}=\frac{1}{1-d(t)\Delta^+g(t)}<0
\end{equation}
if and only if $1<d(t)\Delta^+g(t)$. We will still denote by $\{s_1,\dots,s_N\}=\{t\in[0,T)\cap D_g\ 
:\ 1-d(t)\Delta^+ g(t)<0\}$ and $s_{N+1}=T$, so
\begin{equation}
\widehat {\widetilde{d}}(t)=\begin{dcases} \widetilde d(t)=d(t) & t\in[0,T]\setminus D_g,\\ \frac{\ln|1+\widetilde d(t)\Delta^+ g(t)|}{\Delta^+g(t)}=-\frac{\ln|1-d(t)\Delta^+ g(t)|}{\Delta^+g(t)}, & t\in[0,T)\cap D_g.
\end{dcases}
\end{equation}

As we can see above, the general expression of the exponential $\widehat e_{\widetilde d}(t)$ and, therefore, of the 
solution of the  general linear  equation (\ref{eq:sollingen}), has a convoluted statement. This expression can 
be simplified if we consider the particular case $d=\text{\emph{const.}}$ and $d \Delta^+g(t) < 1$, 
$\forall t\in[a,b)\cap D_g$ --the case that we will consider in the numerical experiments.
\begin{equation}
\begin{aligned}
\widehat e_{\widetilde d}(t)= & \exp\left(\int_{[0,t)}\widehat {\widetilde d}(s)\operatorname{d}\mu_g(s)\right)\\
=&\exp\left(\int_{[0,t)\setminus D_g}\widehat {\widetilde d}(s)\operatorname{d}\mu_g(s)\right)\exp\left(\int_{[0,t)\cap D_g}\widehat {\widetilde d}(s)\operatorname{d}\mu_g(s)\right)\\ = & \exp(d\mu_g([0,t)\setminus D_g))\exp\left(-\sum_{s\in[0,t)\cap D_g}\ln|1-d\Delta^+ g(s)|\right).
\end{aligned}
\end{equation}
Now, by elementary properties of measure spaces,
\begin{equation}
\mu_g([0,t)\setminus D_g) =\mu_g([0,t))-\mu_g([0,t)\cap D_g)= 
g(t)-\sum_{s \in [0,t)\cap D_g} \Delta^+ g(s),
\end{equation}
and we obtain
\begin{equation}
\begin{aligned}
\widehat e_{\widetilde d}(t)=& \exp(d g(t)) \exp\left[- \sum_{s \in [0,t)\cap Dg} \left(d \Delta^+ g(s)+
\ln|1-d\Delta^+ g(s)|\right)
\right] \\
=&\exp(d g(t)) \left[\prod_{s\in [0,t)\cap D_g} \exp(d \Delta^+ g(s)) \, 
|1-d\Delta^+ g(s)| \right]^{-1}.
\end{aligned}
\end{equation}
It is  also remarkable that, in the case of $g(t)=t$, we recover the classical exponential. Also, we have 
the following direct result for the problem with constant coefficients. 

\begin{thm} Let $g:[0,T]\to{\mathbb R}$ be increasing, left continuous and such that $g(0)=0$; $x_0 \in 
\mathbb{R}$, $h\in\mathbb{R} $ and $d\in \mathbb{R}$ such that
\begin{align}
& d \Delta^+g(t) < 1\enskip\forall t\in[a,b)\cap D_g,\\
& \sum_{t\in[a,b)\cap D_g}\left|\ln(1-d\Delta^+g(t))\right|<\infty.
\end{align}
Then the solution of the problem
\begin{equation}\label{pbldcte}\begin{aligned}
x'_g(t)+d x(t)= & h ,\; t\in[0,T),\\
x(0)= & x_0,
\end{aligned}
\end{equation}
is given by the following expression:
\begin{equation} \label{eqpis}
\begin{aligned}
x(t)=&x_0 \exp(-d g(t)) \left[\prod_{s\in [0,t)\cap D_g} \exp(d \Delta^+ g(s)) \, 
(1-d\Delta^+ g(s)) \right] \\
& + h \int_{[0,t)} \exp(d(g(s)-g(t)) \\ & \cdot \left[ \prod_{u\in (s,t)\cap D_g} \exp(d \Delta^+ g(u)) 
(1-d\Delta^+ g(u)) \right]\, \operatorname{d} \mu_g(s).
\end{aligned}
\end{equation}
Observe that this expression satisfies the semigroup property, that is, if $h=0$, then
\begin{equation} \label{eqpis2}
x(t+r)=x(t)\exp(-d\mu_g([t,t+r)\setminus D_g))\prod_{s\in[t,t+r)\cap D_g}(1-d\Delta^+ g(s)).
\end{equation}
\end{thm}

If we assume that the set of discontinuities of function $g$ is finite and we 
consider a time discretization $\{t_k\}_{k=0}^{N+1}\subset [0,T]$ satisfying hypothesis  (H4)
with $d\Delta^+g(t)< 1$, $\forall t\in[a,b)\cap D_g$, then the condition 
$\sum_{t\in[a,b)\cap D_g}\left|\ln\left(1-d\Delta^+g(t)\right)\right|<\infty$ is trivially satisfied. So, 
we have the following corollary for the homogenous case.

\begin{cor}
	Let $g:[0,T]\to{\mathbb R}$ be increasing and left continuous, such that $g(0)=0$ with a set of discontinuity 
	points that we can assume equal to the discretization points, that is, $D_g= \{t_1,\dots,t_N\}\subset(0,T)$, 
	where $t_k<t_{k+1}$ for $k=1,\dots, N-1$. Then, 
	the solution of the problem
	\begin{equation}\label{eqlgb}
	\begin{aligned}
	x'_g(t)+dx(t)= & 0,\; t\in[0,T),\\
	x(0)= & x_0,
	\end{aligned}
	\end{equation}
	where $x_0,d\in{\mathbb R}$ and $d\Delta^+g(t)< 1$, $\forall t\in[a,b)\cap D_g$, is given by
	\begin{equation}\label{sollin}
	x(t)=x_0\exp(-dg(t))\begin{dcases}1, & t\in[0,t_1],\\
	\prod_{k=1}^n(1-d\Delta^+g(t_k))\exp(d\Delta^+g(t_k)), & t\in (t_n,t_{n+1}],\\ & 
	n=1,\dots,N-1,\\
	\prod_{k=1}^N(1-d\Delta^+g(t_k))\exp(d\Delta^+g(t_k)), & t\in (t_N,T].
	\end{dcases}
	\end{equation}
\end{cor}
\begin{proof} Observe that the first case of equation~\eqref{sollin} is just the second case for $n=0$, so we proceed by induction to prove the first and second cases. For $n=0$, taking $x$ as in~\eqref{eqpis}, for $t\in[0,t_1]$ we have that
	\begin{equation}
	x(t)=x_0\exp(-d g(t)).
	\end{equation}
	Assume the result is true for $t\in[0,t_{n}]$ with $n\in\{1,\dots,N-1\}$. For $t\in(t_{n},t_{n+1}]$ 
	using the semigroup property~\eqref{eqpis2},
	\begin{equation}\begin{aligned}
	x(t)= & x(t_n+t-t_n)=x(t_n)\exp(-d\mu_g([t_n,t)\setminus D_g))
	\prod_{s\in[t_n,t)\cap D_g}\left(1-d\Delta^+ g(s)\right) \\
	= & x(t_n)\exp(-d(g(t)-g(t_n))) \exp(d \Delta^+ g(t_n)) \left(1-d\Delta^+ g(t_n)\right)
	\\= &x_0\exp(-dg(t))\prod_{k=1}^n\left(1-d\Delta^+g(t_k)\right)\exp(d\Delta^+g(t_k)).
	\end{aligned}
	\end{equation}
	The third case of~\eqref{sollin} is straightforward from the previous one.
\end{proof}

Finally, it is remarkable that in previous corollary we can change the hypothesis $d\Delta^+g(t)< 1$, $\forall t\in[a,b)\cap D_g$ 
by $d\Delta^+g(t) \neq 1$, $\forall t\in[a,b)\cap D_g$, and obtain a similar expression for the solution taking 
into account the general formula (\ref{eq:sollingen}). The last hypothesis is more general that the previous one but, 
in order to present the results in a clear way, we will assume that the first hypothesis is fulfilled. 

\section{Numerical simulations}\label{NS}

In this section we will present some numerical results that we have reached using 
the scheme (\ref{eq:ec11}) for approximating the solution of the homogeneous 
linear equation (\ref{eqlgb}) with constant coefficients. We will also compare 
the numerical solution with the explicit solution (\ref{eqlgb}) that we have obtained 
in the previous section. Finally, to test the robustness of the method, we will use the 
numerical scheme to approximate the solution of a silkworm population model based 
on the example presented in \cite{POUSO2018}.

\subsection{Approximation of the general linear equation}

In order to validate the scheme (\ref{eq:ec11}) for different number of 
discontinuities in the derivator $g$ (the main difficulty of the problem), we will 
consider an increasing regular continuous part $g^C$ and we will obtain several 
$g$ test functions summing to the previous one the jump part $g^B$ 
associated to several choices of jumps. We consider the following function:
\begin{equation}
\varphi:x \in \mathbb{R} \rightarrow \varphi(x)=
\left\{\begin{array}{ll}
0, & x\leq 0, \\
\displaystyle \left[1+\exp\left(
-2\alpha \tan\left(\frac{\pi}{2}(2x-1) \right)
\right)
\right]^{-1}, & x \in (0,1), \\
1, & x \geq 1,
\end{array}\right.
\end{equation}
where $\alpha>0$. We have that $\varphi$ is a increasing $\mathcal{C}^{\infty}(\mathbb{R})$ function 
and we can use it to construct a more sophisticated increasing function $g^C$ that will be constant in 
some intervals. For instance, for $\alpha=4$, we can consider the following function $g^C$ 
in the time interval $[0,10]$ and from it build the derivator $g$ by adding the jump function $g^B$:
\begin{figure}[H]
\begin{subfigure}{.5\textwidth}
\centering
\includegraphics[width=1\linewidth]{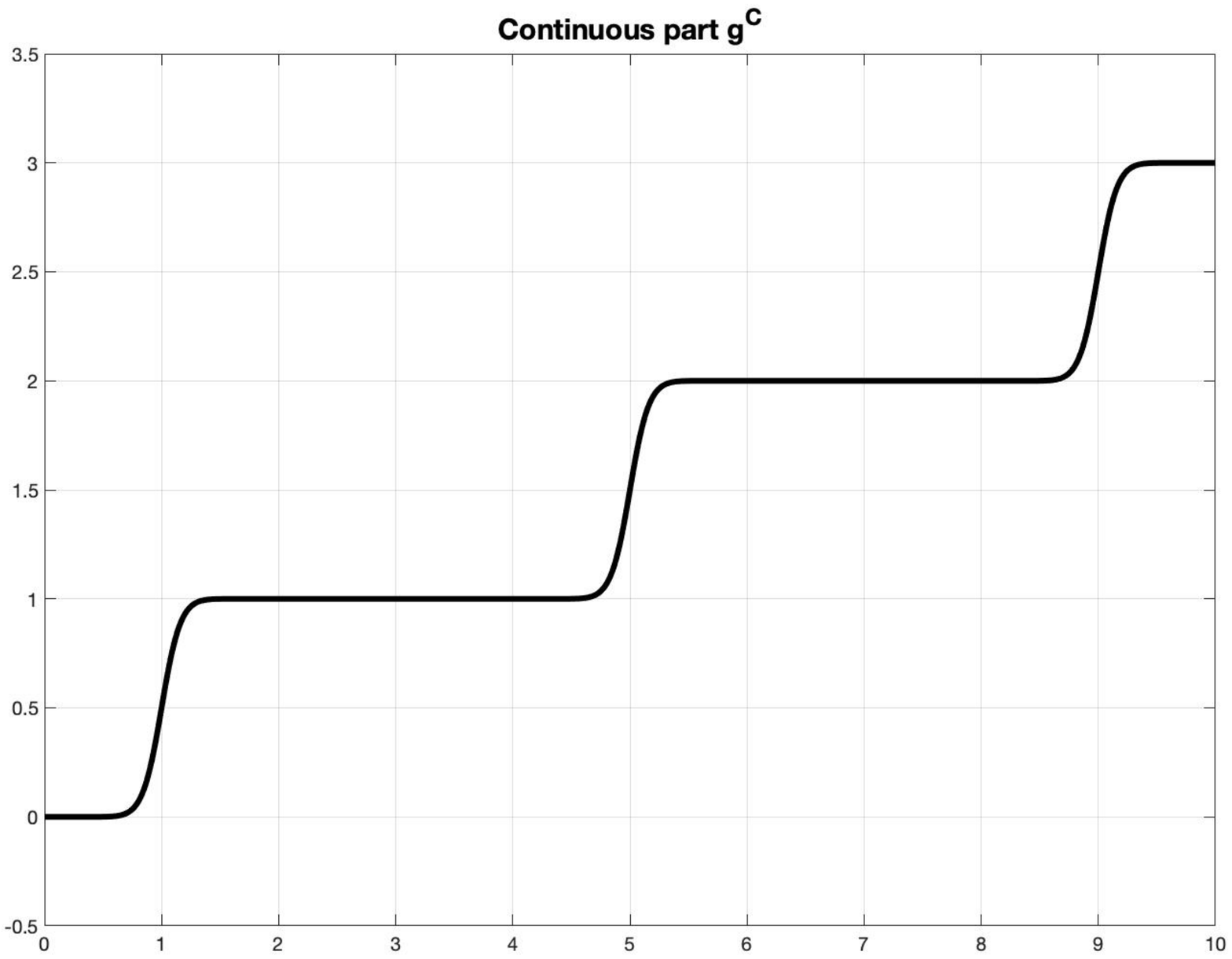}
\caption{Continuous part $g^C$.}
\label{fig:figure1}
\end{subfigure}
\begin{subfigure}{.5\textwidth}
\centering
\includegraphics[width=1\linewidth]{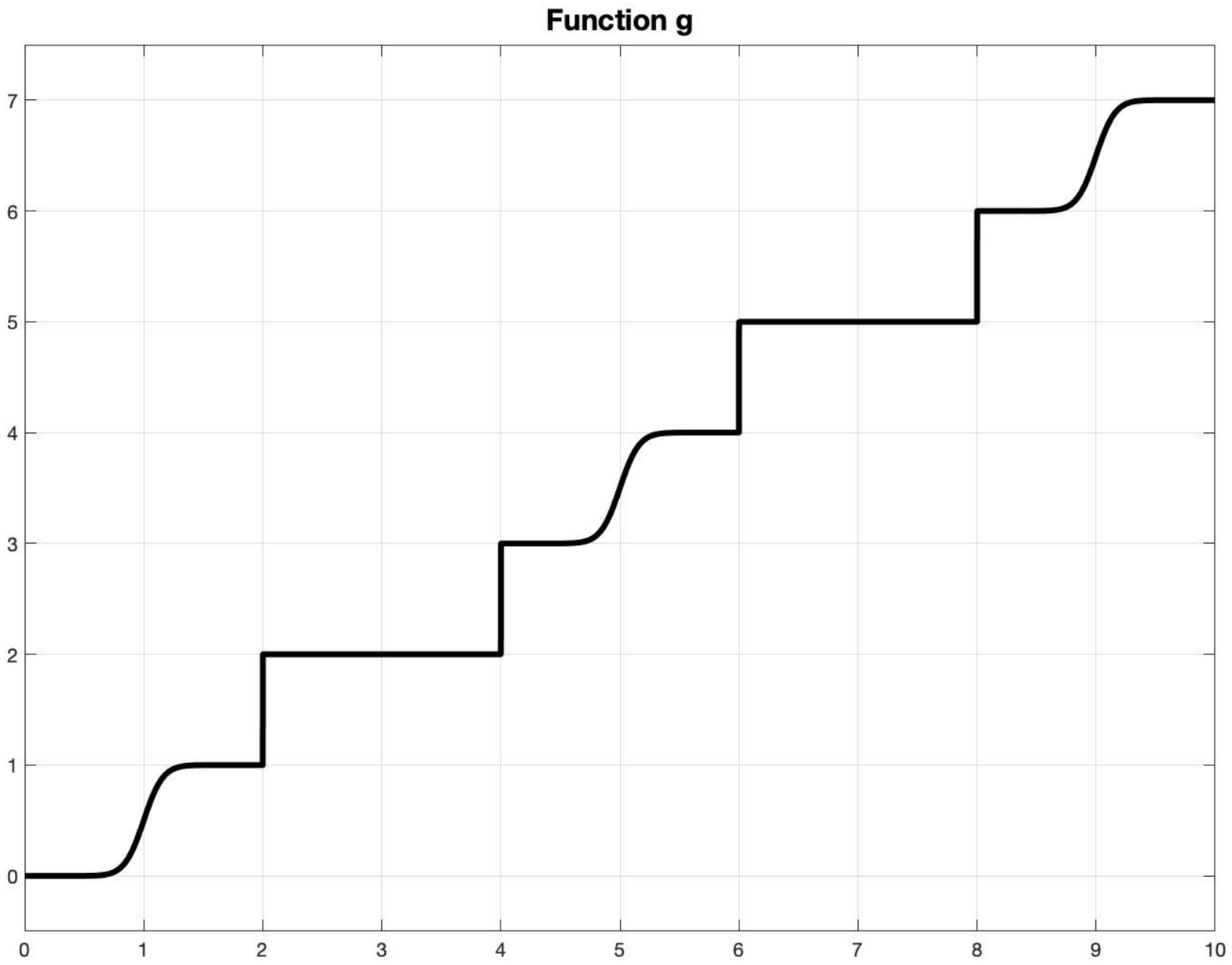}
\caption{Derivator $g$.}
\label{fig:figure2}
\end{subfigure}
\caption{Example of continuous part $g^C$ and derivator $g$.}
\end{figure}
In Figure~\ref{fig:figure1} we observe that we have concatenated three times the function $\varphi$ and, in order 
to obtain the derivator function $g$, we have added four jumps at the times $\tilde{t}_k=2k$, $k=1,2,3$, 
with $\Delta^+ g(\tilde{t}_k)=1$. In Figure~\ref{fig:figure5} we plot the solution for $d=-0.5$ and 
in Figure~\ref{fig:figure6} the solution for $d=0.5$. As we can see in both figures, we have inactivity 
periods where the function $g$ is constant and impulses in the times where the function $g$ 
presents discontinuities.
\begin{figure}[H]
\begin{subfigure}{.5\textwidth}
\centering
\includegraphics[width=1\linewidth]{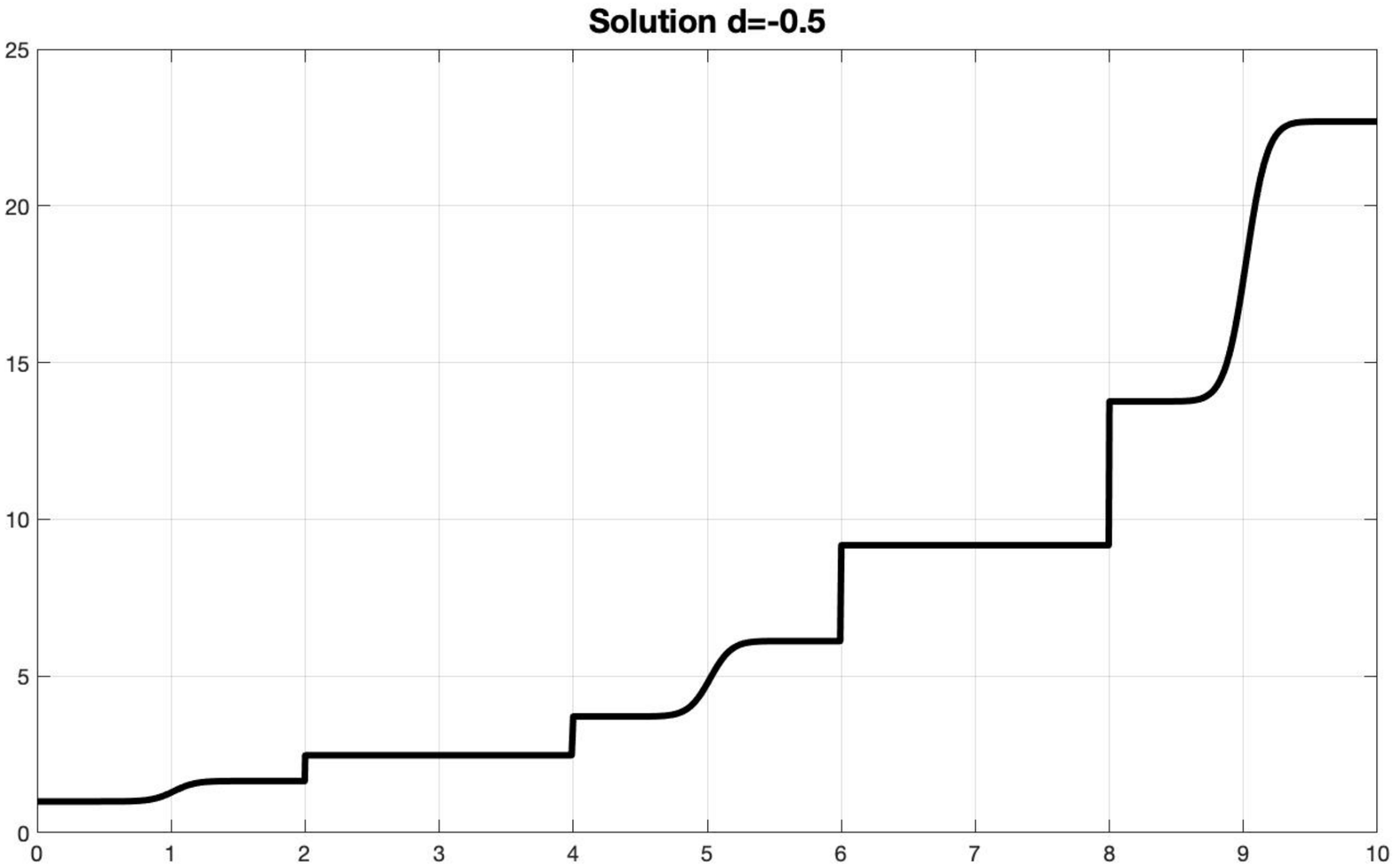}
\caption{Solution for $d=-0.5$.}
\label{fig:figure5}
\end{subfigure}
\begin{subfigure}{.5\textwidth}
\centering
\includegraphics[width=1\linewidth]{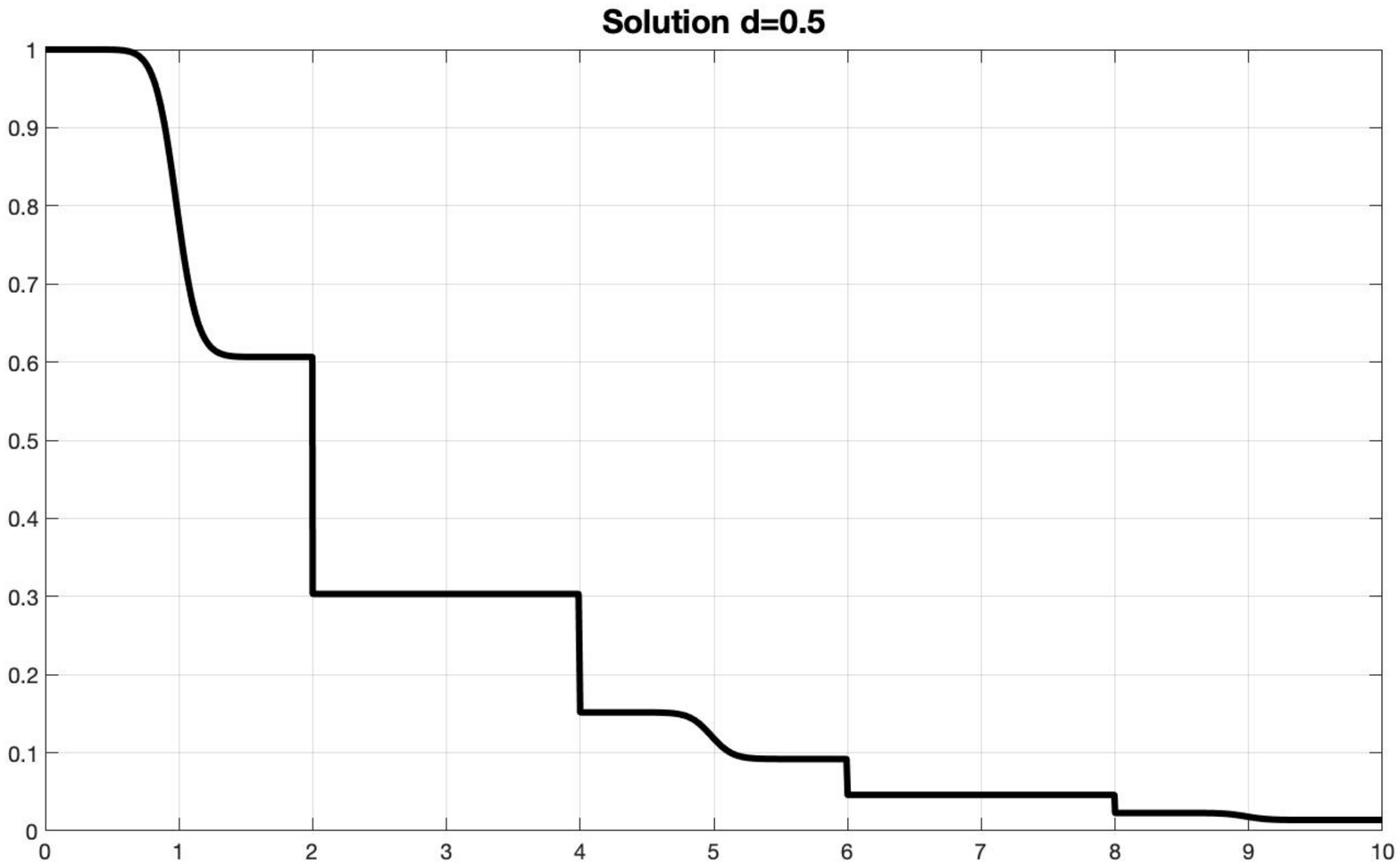}
\caption{Solution for $d=0.5$.}
\label{fig:figure6}
\end{subfigure}
\caption{Explicit solution for different values of $d$ ($x_0=1$).}
\end{figure}

We summarize the results 
obtained for different values of time step $h$ taking $x_0=1$, $d=-0.5$, 
and for different values of $\#D_g$, with $\Delta^+ g(s)=1$, 
$\forall s \in D_g$:
\begin{table}[H]
\centering
\noindent\resizebox{\textwidth}{!}{\begin{tabular}{|r|c|c|c|c|c|}
\hline
$\#D_g=2$ & $h=1.e-01$ & $h=1.e-02$ & $h=1.e-03$ & $h=1.e-04$ & $h=1.e-05$ \\
\hline 
$\max\{|e_n^*|\}$ 
& $1.1704e-01$
& $1.2136e-03$
& $1.2100e-05$
& $1.2095e-07$
& $1.2108e-09$\\
\hline
$\max\{|e_n|\}$ 
& $3.1399e-02$
& $3.3911e-04$ 
& $3.4002e-06$
& $3.4010e-08$
& $3.4173e-10$\\
\hline
$\max\{|e_n^+|\}$ 
& $1.2573e-02$
& $1.3697e-04$
& $1.3747e-06$
& $1.3752e-08$
& $1.3830e-10$\\
\hline
\hline
$\#D_g=4$ & $h=1.e-01$ & $h=1.e-02$ & $h=1.e-03$ & $h=1.e-04$ & $h=1.e-05$ \\
\hline 
$\max\{|e_n^*|\}$ 
& $2.6454e-01$
& $2.7321e-03$
& $2.7226e-05$
& $2.7213e-07$
& $2.7202e-09$\\
\hline
$\max\{|e_n|\}$ 
& $7.2094e-02$
& $7.6469e-04$
& $7.6522e-06$
& $7.6523e-08$
& $7.6426e-10$\\
\hline
$\max\{|e_n^+|\}$ 
& $2.9167e-02$
& $3.0921e-04$
& $3.0942e-06$
& $3.0942e-08$
& $3.0848e-10$\\
\hline
\hline
$\#D_g=6$ & $h=1.e-01$ & $h=1.e-02$ & $h=1.e-03$ & $h=1.e-04$ & $h=1.e-05$ \\
\hline 
$\max\{|e_n^*|\}$ 
& $6.1870e-01$
& $6.3567e-03$
& $6.3285e-05$
& $6.3250e-07$
& $6.3250e-09$ \\
\hline
$\max\{|e_n|\}$ 
& $1.9034e-01$
& $1.9793e-03$
& $1.9749e-05$
& $1.9744e-07$
& $1.9764e-09$\\
\hline
$\max\{|e_n^+|\}$ 
& $8.2704e-02$
& $8.5267e-04$
& $8.4975e-06$
& $8.4941e-08$
& $8.4977e-10$\\
\hline
\hline
$\#D_g=8$ & $h=1.e-01$ & $h=1.e-02$ & $h=1.e-03$ & $h=1.e-04$ & $h=1.e-05$ \\
\hline 
$\max\{|e_n^*|\}$ 
& $1.3225e+00$
& $1.3563e-02$
& $1.3510e-04$
& $1.3503e-06$
& $1.3500e-08$\\
\hline
$\max\{|e_n|\}$ 
& $3.4486e-01$
& $3.5400e-03$
& $3.5324e-05$
& $3.5315e-07$
& $3.5304e-09$\\
\hline
$\max\{|e_n^+|\}$ 
& $1.3545e-01$
& $1.3644e-03$
& $1.3593e-05$
& $1.3587e-07$
& $1.3521e-09$\\
\hline
\hline
$\#D_g=10$ & $h=1.e-01$ & $h=1.e-02$ & $h=1.e-03$ & $h=1.e-04$ & $h=1.e-05$ \\
\hline 
$\max\{|e_n^*|\}$ 
& $2.9049e+00$
& $2.9828e-02$
& $2.9723e-04$
& $2.9708e-06$
& $2.9703e-08$\\
\hline
$\max\{|e_n|\}$ 
& $6.9124e-01$
& $7.1152e-03$
& $7.1039e-05$
& $7.1025e-07$
& $7.1110e-09$\\
\hline
$\max\{|e_n^+|\}$ 
& $2.5333e-01$
& $2.5546e-03$
& $2.5465e-05$
& $2.5456e-07$
& $2.5464e-09$\\
\hline
\end{tabular}}
\caption{Numerical results (linear equation)}
\label{tab:table1}
\end{table} 

From the table above we can observe that numerical errors grow as the number of discontinuities in 
the derivator increases. This behavior is consistent with the error bounds obtained in Theorem~\ref{thmcon}
in which the term $[1+G_2]^{\#D_g}$ appears multiplying the error expressions. In Figure~\ref{fig:figure3} 
we can observe the error evolution for the predictor and, in Figure~\ref{fig:figure4}, the error 
evolution for the corrector. We realize that the global behavior in terms of $h$ for the predictor is 
$O(h)$ and $O(h^2)$ for the corrector. This improvement in the order of convergence 
with respect to the one predicted in theory is a consequence of the fact that, 
thanks to the regularity of the solution, the trapezoidal formula is more accurate.
\begin{figure}[H]
\begin{subfigure}{.5\textwidth}
\centering
\includegraphics[width=1\linewidth]{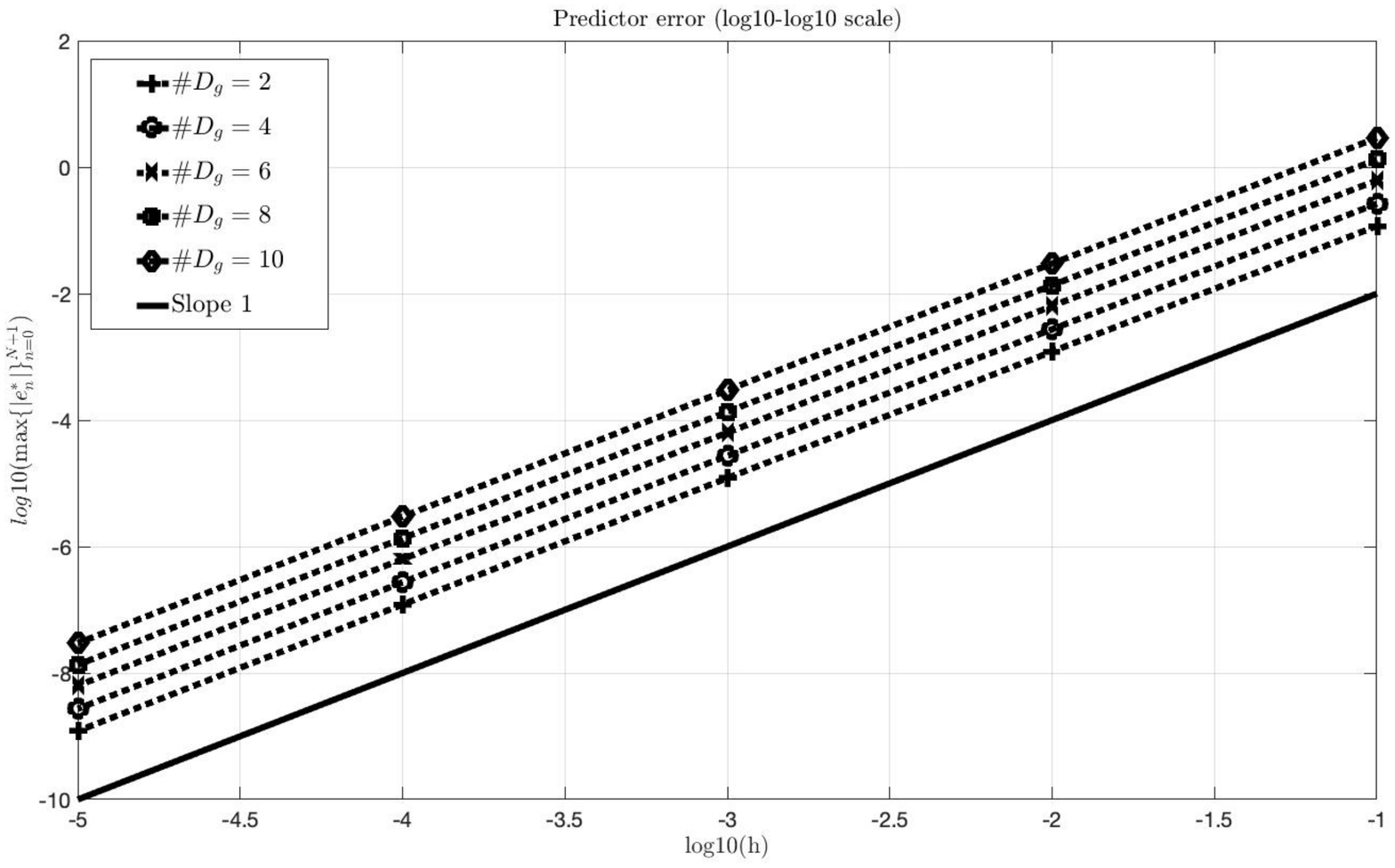}
\caption{Predictor error in $log10-log10$ scale}
\label{fig:figure3}
\end{subfigure}
\begin{subfigure}{.5\textwidth}
\centering
\includegraphics[width=1\linewidth]{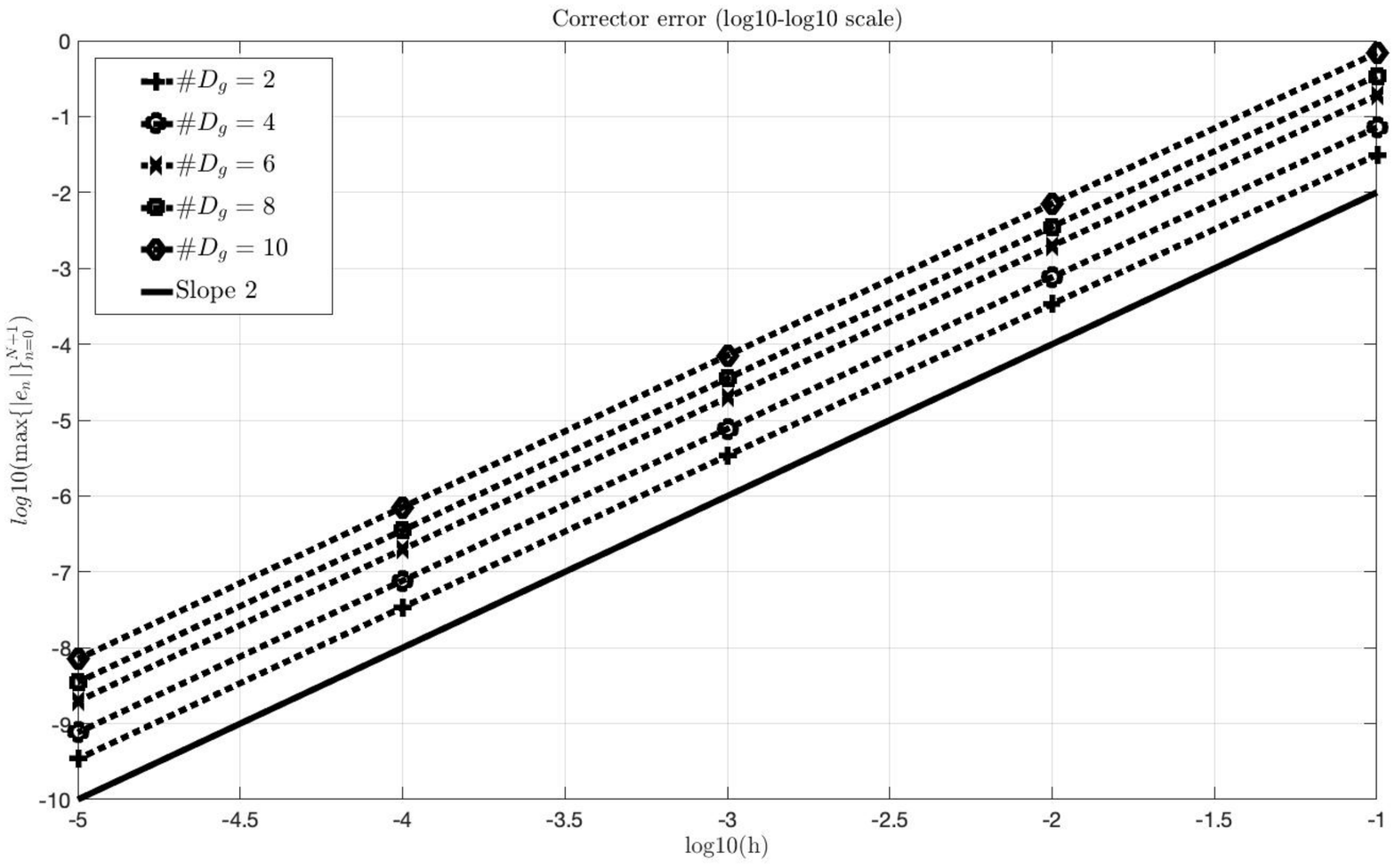}
\caption{Corrector error in $log10-log10$ scale}
\label{fig:figure4}
\end{subfigure}
\caption{Evolution of the global error}
\end{figure}

\subsection{Approximation of a silkworm population model}

We present in this section the numerical approximation of a realistic 
case which corresponds to a silkworm population model based 
on the example presented in \cite{POUSO2018}, that we will briefly summarize 
for the convenience of the reader. In this example the authors 
consider that the life cycle of silkworms has three stages: worm, cocoon and moth. 
Moths lay eggs and die soon after, then eggs hatch and produce a completely new 
colony of silkworms.
\begin{equation}
\begin{array}{|l|l|}\hline \text { Stage } & {\text { Time Intervals }} \\ \hline \text { Worms } & {(5 k, 5 k+2], k=0,1,2, \ldots} \\ \hline \text { Cocoons } & {(5 k+2,5 k+3], k=0,1,2, \ldots} \\ \hline \text { Moths } & {(5 k+3,5 k+4], k=0,1,2, \ldots} \\ \hline \text { Eggs } & {(5 k+4,5 k+5], k=0,1,2, \ldots} \\ \hline\end{array}
\end{equation}
In order to take into account the previous behavior, they consider the 
following derivator $g:[0,\infty) \rightarrow \mathbb{R}$:
\begin{equation}
g(t)=\begin{dcases}{\frac{1}{2} \sqrt{4 t-t^{2}},} & {\text { if } 0 \leqslant t \leqslant 2}, \\ {1,} & {\text { if } 2<t \leqslant 3}, \\ {2-\sqrt{6 t-t^{2}-8},} & {\text { if } 3<t \leqslant 4}, \\ {3,} & {\text { if } 4<t \leqslant 5}, \\
4+g(t-5), & {\text { if } 5 > t,}
\end{dcases}
\end{equation}
and they solve the following Stieltjes equation:
\begin{equation} \label{eq:stielt1}
\left\{\begin{array}{l}
\displaystyle 
x_g'(t)=f(t,x(t),x),\; t\in (0,\infty)\setminus D_g,\\
x(0)=x_0.
\end{array}\right.
\end{equation}
Where $f:[0,T] \setminus C_g \times \mathbb{R} \times L^1_{\operatorname{loc}}(\mathbb{R}) \rightarrow 
\mathbb{R}$ is such that
\begin{equation} \label{eq:stielt2}
f(t, x, \varphi)=\begin{dcases}-c x, & {\text { if } t \in(5 k, 5 k+4), k=0,1,2, \ldots} \\ 
{-x,} & {\text { if } t=5 k+4, k=0,1,2, \ldots} \\ {\lambda \int_{t-5}^{t-1} \varphi(s) \operatorname{d} s,} & {\text { if } t=5(k+1), k=0,1,2, \ldots}\end{dcases}
\end{equation}
with $c>0$, $\lambda>0$. In  \cite[Proposition 5.1]{POUSO2018} the authors obtain the 
explicit solution of the previous model:
\begin{equation}
x(t)=\left\{\begin{array}{ll}
\displaystyle x_0 \exp(-c g(t)),& 0\leq t\leq 4, \\
\displaystyle \lambda \exp\left(-c(g(t)-g(5k^+))\right)
\int_{5(k-1)}^{5k-1} x(s)\, ds
, & 5k<t\leq 5k+4,\; k\in \mathbb{N}, \\
0, & \text{otherwise}.
\end{array}\right.
\end{equation}

Now, we approximate the solution of the Stieltjes differential equation~\eqref{eq:stielt1} using 
the scheme (\ref{eq:ec11}) in the time interval $[0,10]$, for $\lambda=1.1$, $c=1.2$ and $x_0=8$. 
We realize that in order to evaluate the function \eqref{eq:stielt2} we have to approximate the 
integral value using a classical quadrature formulae, so the convergence order of the full scheme 
will be penalized by this approximation. In our case we have considered a composite 
trapezoidal rule. In the following table we summarize the numerical results that we have 
obtained in this case (we omit the errors for the predictor and the limits from the right):
\begin{table}[H]
\centering
\noindent\resizebox{\textwidth}{!}{\begin{tabular}{|r|c|c|c|c|c|}
\hline
& $h=1.e-01$ & $h=1.e-02$ & $h=1.e-03$ & $h=1.e-04$ & $h=1.e-05$ \\
\hline
$\max\{|e_n|\}$ 
& $2.3724e-01$ 
& $1.7138e-02$
& $4.8860e-03$
& $1.5287e-03$
& $4.8291e-04$
\\ \hline
\end{tabular}}
\caption{Numerical results (silkworm population model).}
\label{tab:table2}
\end{table} 
Finally, in Figure~\ref{fig:figure7} we can see the exact solution and the 
predictor using as time step $h=1.e-01$ and, in Figure~\ref{fig:figure8}, 
 $h=1.e-05$.
\begin{figure}[H]
\begin{subfigure}{.5\textwidth}
\centering
\includegraphics[width=1\linewidth]{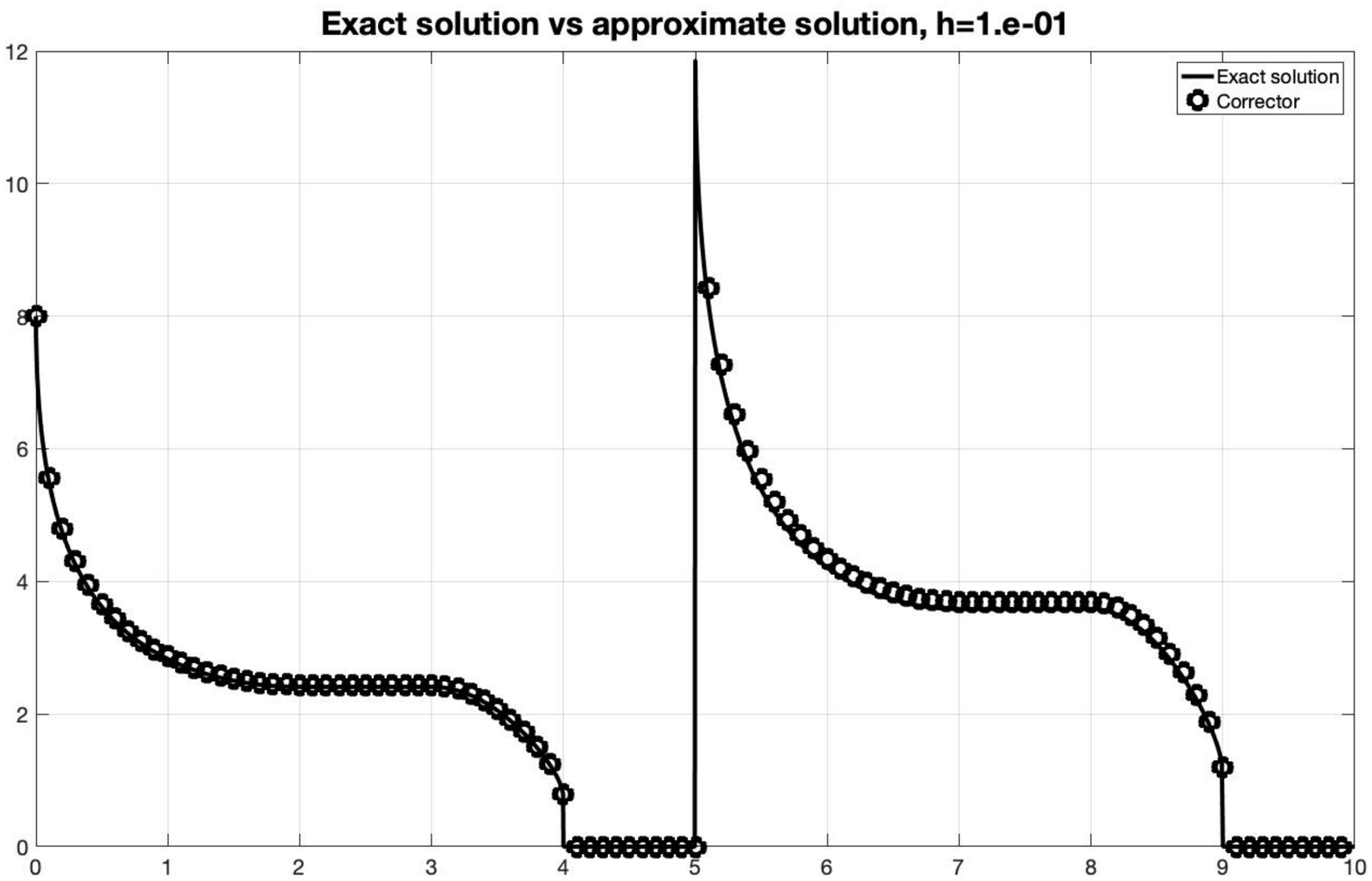}
\caption{$h=1.e-01$.}
\label{fig:figure7}
\end{subfigure}
\begin{subfigure}{.5\textwidth}
\centering
\includegraphics[width=1\linewidth]{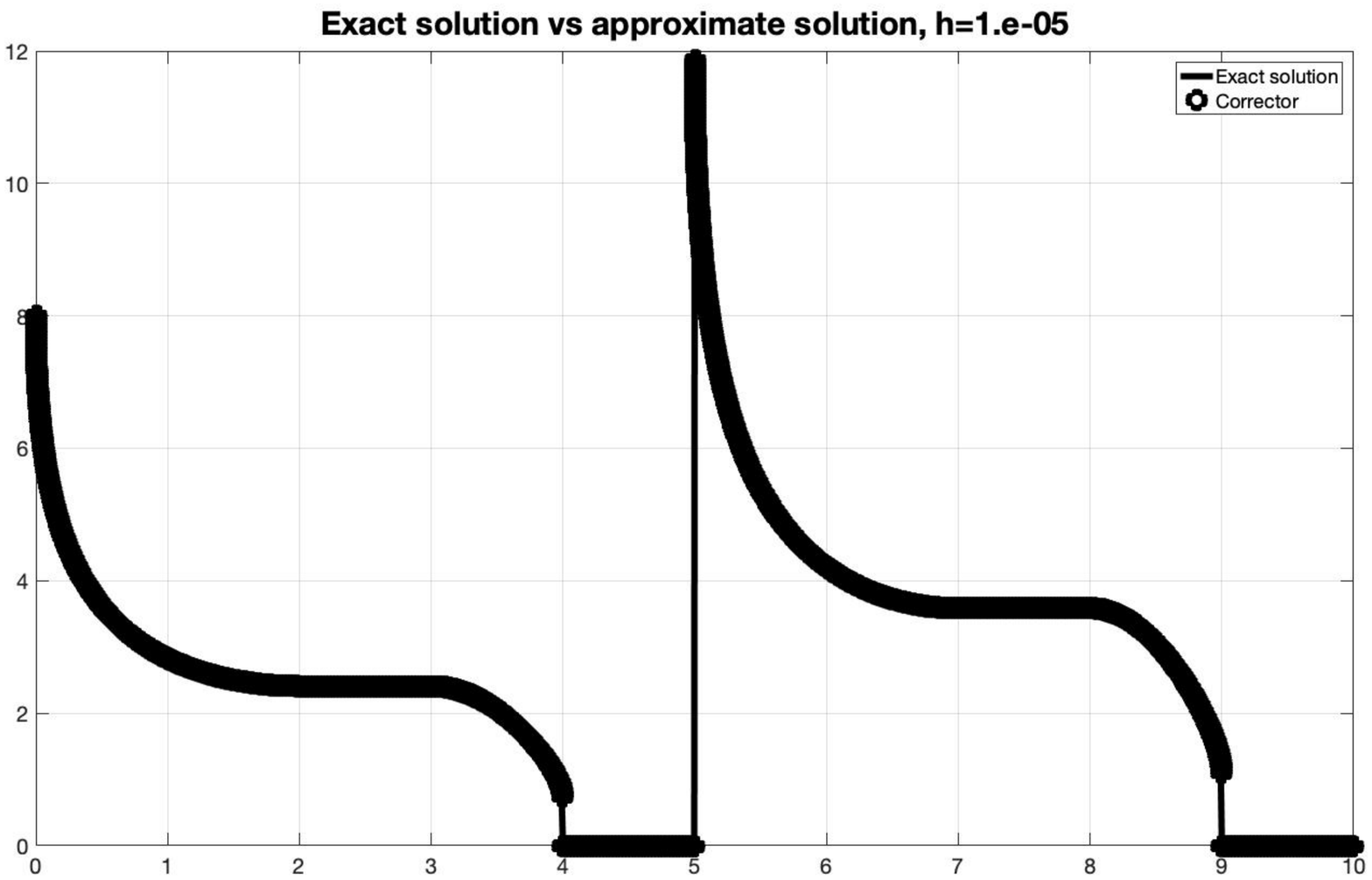}
\caption{$h=1.e-05$.}
\label{fig:figure8}
\end{subfigure}
\caption{Exact solution vs approximate solution.}
\end{figure}

\section*{Acknowledgements}
	
F. Adrián F. Tojo would like to acknowledge his gratitude towards Prof. Stefano Bianchini, whose comments regarding Lipschitz functions allowed to improve Corollary~\ref{corld}.

\bibliography{Bibliografia}

\end{document}